\documentclass[a4paper]{amsart}
\usepackage[utf8]{inputenc}
\usepackage[T1]{fontenc}
\usepackage{lmodern, enumerate}
\usepackage{amssymb,amsxtra}
\usepackage{amscd}
\usepackage[all]{xy}
\usepackage{xcolor}
\usepackage{nicefrac,mathtools}
\usepackage{microtype}
\usepackage[pdftitle={Universal and exotic generalised fixed-point algebras},
 pdfauthor={Alcides Buss, Siegfried Echterhoff},
 pdfsubject={Mathematics}
]{hyperref}
\usepackage[lite]{amsrefs}
\newcommand*{\MRref}[2]{ \href{http://www.ams.org/mathscinet-getitem?mr=#1}{MR #1}}
\newcommand*{\arxiv}[1]{\href{http://www.arxiv.org/abs/#1}{arXiv: #1}}

\numberwithin{equation}{section}
\theoremstyle{plain}
\newtheorem{theorem}[equation]{Theorem}
\newtheorem{lemma}[equation]{Lemma}
\newtheorem{proposition}[equation]{Proposition}
\newtheorem{corollary}[equation]{Corollary}
\theoremstyle{definition}
\newtheorem{definition}[equation]{Definition}
\theoremstyle{remark}
\newtheorem{remark}[equation]{Remark}
\newtheorem{example}[equation]{Example}

\DeclareMathOperator{\Aut}{Aut}
\DeclareMathOperator{\cspn}{\overline{span}}
\DeclareMathOperator{\spn}{{span}}

\DeclareMathOperator{\id}{\mathrm{id}}

\DeclareMathOperator{\pt}{\mathrm{pt}}
\DeclareMathOperator{\End}{\mathrm{End}}

\newcommand*{\nb}{\nobreakdash}
\newcommand*{\Star}{\(^*\)\nobreakdash-}

\newcommand*{\R}{\mathbb R}
\newcommand*{\T}{\mathbb T}
\newcommand*{\Z}{\mathbb Z}
\newcommand*{\N}{\mathbb N}
\newcommand*{\C}{\mathbb C}
\newcommand*{\Q}{\mathbb Q}
\newcommand*{\Lb}{\mathcal L}
\renewcommand*{\L}{\mathcal L}
\newcommand*{\K}{\mathcal K}
\renewcommand*{\H}{\mathcal H}

\newcommand*{\F}{\mathcal F} 
\newcommand*{\cont}{C}
\newcommand*{\contc}{\cont_c}
\newcommand*{\Mat}{\mathbb M}
\newcommand*{\M}{\mathcal M}
\newcommand*{\Rel}{\mathcal R}

\newcommand*{\Ad}{\textup{Ad}}

\newcommand*{\U}{\mathcal U}
\newcommand*{\E}{\mathcal E}
\newcommand*{\defeq}{\mathrel{\vcentcolon=}}

\newcommand*{\ket}[1]{\lvert#1\rangle}
\newcommand*{\bra}[1]{\langle#1\rvert}
\newcommand*{\braket}[2]{\langle#1\!\mid\!#2\rangle}

\newcommand*{\bbraket}[2]{\mathopen{\langle\!\langle}#1\!\mid\!#2\mathclose{\rangle\!\rangle}}
\newcommand*{\kket}[1]{{\lvert#1\mathclose{\rangle\!\rangle}}}
\newcommand*{\bbra}[1]{{\mathopen{\langle\!\langle}#1\rvert}}

\newcommand*{\sbe}{\subseteq} 

\newcommand*{\cstar}{\texorpdfstring{$C^*$\nobreakdash-\hspace{0pt}}{*-}}
\newcommand*{\into}{\hookrightarrow}

\newcommand*{\red}{r}

\newcommand*{\un}{u}

\newcommand*{\Fix}{\mathrm{Fix}}
\newcommand*{\dual}[1]{\widehat{#1}}
\newcommand*{\dualG}{\widehat{G}}

\newcommand{\om}{\omega}

\newcommand*{\su}{\mathrm{su}}

\newcommand*{\vN}{\operatorname{vN}}
\newcommand{\si}{\mathrm{si}}
\newcommand{\ii}{\mathrm{i}}

\newcommand{\eg}{\emph{e.g.~}}

\newcommand*{\Sch}{\mathcal S} 

\begin{document}
\title[Rieffel proper actions]{Rieffel proper actions}

\author{Alcides Buss}
\email{alcides.buss@ufsc.br}
\address{Departamento de Matem\'atica\\
 Universidade Federal de Santa Catarina\\
 88.040-900 Florian\'opolis-SC\\
 Brazil}

\author{Siegfried Echterhoff}
\email{echters@math.uni-muenster.de}
\address{Mathematisches Institut\\
 Westf\"alische Wilhelms-Universit\"at M\"un\-ster\\
 Einsteinstr.\ 62\\
 48149 M\"unster\\
 Germany}

\begin{abstract}
In the late 1980's Marc Rieffel introduced a notion of properness for actions
of locally compact groups on C*-algebras which, among other things, allows
the construction of generalised fixed-point algebras for such actions.
In this paper we give a simple characterisation of Rieffel proper actions and
use this to obtain several (counter) examples for the
theory. In particular, we provide examples of Rieffel proper actions $\alpha:G\to\Aut(A)$ for which
properness is not induced by a nondegenerate equivariant *-homomorphism $\phi:C_0(X)\to \M(A)$
for any proper $G$-space $X$. Other examples, based on earlier work of Meyer,  show that a given action might carry different
structures for Rieffel properness with different generalised fixed-point algebras.
\end{abstract}



\thanks{Supported by Deutsche Forschungsgemeinschaft  (SFB 878, Groups, Geometry \& Actions) and by PVE/CAPES \& CNPq -- Brazil.}
\maketitle

\section{Introduction}\label{sec:introduction}
In \cite{Rieffel:Proper} Marc Rieffel introduced a notion of proper actions (which we call
{\em Rieffel proper actions} below) of a group $G$ on a
\cstar{}algebra $A$ which allows the construction of a generalised fixed-point algebra $A^G$ together with a natural Morita equivalence bimodule between this algebra and a suitable ideal of the reduced crossed product $A\rtimes_{\alpha,r} G$.
Rieffel's notion of properness depends on a choice of a dense \Star{}subalgebra $A_0$ of $A$ which
must satisfy a number of quite technical conditions (see \S\ref{sec-Rieffel-proper} below).
One of these conditions requires that
for all $\xi,\eta\in A_0$ the functions
$$t\mapsto \bbraket {\xi}{\eta}(t)\defeq\Delta(t)^{-1/2}\xi^*\alpha_t(\eta)$$
lie in $L^1(G,A)\subseteq A\rtimes_{\alpha,r}G$.
Rieffel's conditions  allow the construction of a corresponding generalised fixed-point algebra $A^G\subseteq \M(A)$ and an
equivalence bimodule $\F(A_0)$ between $A^G$ and the closed ideal $\mathcal{I}_{A_0}\subseteq A\rtimes_r G$
generated by all elements of the form $\{\bbraket{\xi}{\eta}: \xi,\eta\in A_0\}$.

In this paper we show that an action $\alpha:G\to \Aut(A)$ is Rieffel proper if and only if there exists a dense
\emph{subspace} (not necessarily a subalgebra) $\mathcal R\subseteq A$ which satisfies the following single condition:
\begin{enumerate}
\item[{\bf(P1)}] For all $\xi,\eta\in \mathcal R$ the functions $t\mapsto \xi^*\alpha_t(\eta)$ and $t\mapsto \Delta(t)^{-1/2}\xi^*\alpha_t(\eta)$
belong to $L^1(G,A)$.
\end{enumerate}
If $A_0\subseteq A$ is a dense \Star{}subalgebra which satisfies Rieffel's original conditions, it also
satisfies (P1). We show that, conversely, if $\mathcal R$ is as above, then there is a canonical construction of
a dense \Star{}subalgebra $A_{\mathcal R}$ which satisfies Rieffel's conditions. As easy corollaries  we get the following useful results:
\begin{enumerate}
\item Assume $A$ and $B$ are $G$\nb-algebras such that there exists a nondegenerate $G$\nb-equivariant
\Star{}homomorphism $\phi:A\to\M(B)$. Then, if $A$ is Rieffel proper, so is $B$.
\item If $A$ is a Rieffel proper $G$\nb-algebra and $B$ is a Rieffel proper $H$-algebra, then
$A\otimes_\nu B$ is a Rieffel proper $G\times H$-algebra, where $\otimes_\nu$ might denote
the minimal or maximal tensor product.
\end{enumerate}
These basic results seem to have been not noticed for general Rieffel proper $G$\nb-al\-ge\-bras, although
the first of these results is well known in the case $A=C_0(X)$ for some proper $G$\nb-space $X$. Indeed, most standard examples of Rieffel proper actions of a group $G$ on a C*-algebra $B$, like dual actions of groups on crossed products by coactions, come naturally equipped with  a nondegenerate $G$\nb-equivariant \Star{}homomorphism
$\phi:C_0(X)\to \M(B)$ for some proper $G$\nb-space $X$, and actions with this extra property
have been studied extensively in the literature (e.g., see \cites{Huef-Kaliszewski-Raeburn-Williams:Naturality_Rieffel, anHuef-Raeburn-Williams:Symmetric, Buss-Echterhoff:Exotic_GFPA, Buss-Echterhoff:Imprimitivity, Buss-Echterhoff:Mansfield}).
Following \cites{Buss-Echterhoff:Exotic_GFPA, Buss-Echterhoff:Imprimitivity, Buss-Echterhoff:Mansfield}, we shall call such actions to be {\em weakly proper}.
It has been shown in \cite{Buss-Echterhoff:Exotic_GFPA} that weakly proper actions enjoy many properties
which are (so far) unknown for general Rieffel proper actions. The most remarkable one is that
they allow analogous constructions of the Hilbert $A\rtimes_{\alpha,r}G$-module $\F(A_0)$ for
the universal crossed products $A\rtimes_{\alpha}G:=A\rtimes_{\alpha,\un}G$ and of corresponding universal
generalised fixed-point algebras
$A_\un^G$ with many interesting properties. Looking at the vast number of examples
of weakly proper actions, we were wondering, whether every Rieffel proper action is also
weakly proper.

In \S\ref{sec-twisted} we show that this is not the case. Using our characterisation of Rieffel proper actions
together with Rieffel's deformation $C_0(\R^n)_J$ of $C_0(\R^n)$ by a skew-symmetric matrix $J\in {\Mat_n}(\R)$,
we show that the dual action of the Pontrjagin dual $\widehat G$ of an abelian locally compact group
$G$ on any twisted group algebra $C^*(G,\om)$ attached to any $2$-cocycle $\om\in Z^2(G,\T)$ is Rieffel proper.
On the other hand, if $G$ is connected, we can show that this dual action is weakly proper
only if $\om$ is similar to the trivial cocycle. This shows that there are many natural examples of Rieffel proper actions
which are not weakly proper.

In the final section \S 4 we study the question whether  the generalised fixed-point algebra $A^G$
for a Rieffel proper action  $\alpha:G\to \Aut(A)$ is independent of the choice of the dense subalgebra $A_0\subseteq A$ (or the dense
subspace $\mathcal R\subseteq A$ of our condition (P1)).
Indeed, examples for a dependence on similar
structures have been constructed already by Ralf Meyer in the setting of ``continuously square-integrable actions'', which we here call ``Exel-Meyer proper actions''; these are based on the theory of square-integrable actions
 (see \cites{Exel:SpectralTheory, Meyer:Equivariant, Meyer:generalised_Fixed}) and  generalise Rieffel proper actions.
 Using our main result, we show that many of  Meyer's examples are also Rieffel proper,
hence also provide  examples for the dependence of the fixed-point algebra $A^G$ on the choice of
the dense subalgebra $A_0$ in this setting.
To our knowledge, this provides the first examples for this dependence in the
setting of Rieffel proper actions.

The authors are grateful to Sergey Neshveyev for some helpful discussions on Rieffel deformation theory that led us to the examples in \S\ref{sec-twisted}.

Most of this work has been carried out through a visit of the second author at the Department of Mathematics of the Federal University of Santa Catarina. He thanks the members of the department for their warm hospitality and CAPES for making this visit possible.

\section{Rieffel proper actions}\label{sec-Rieffel-proper}
Suppose $G$ is a locally compact group
acting by a strongly continuous homomorphism $\alpha:G\to\Aut(A)$ on the C*-algebra $A$.
Then, in \cite{Rieffel:Proper}*{Definition 1.2}, Rieffel defines this action to be {\em proper} (which we call {\em Rieffel proper})
if there exists a dense $G$\nb-invariant \Star{}subalgebra $A_0$ of $A$ such that for all $\xi,\eta\in A_0$:
\begin{enumerate}
\item[{\bf (P1)}] the functions $t\mapsto \Delta(t)^{-1/2}\xi^*\alpha_t(\eta)$
and $t\mapsto \xi^*\alpha_t(\eta)$ belong to $L^1(G,A)$;
\item[{\bf (P2)}] there is a (necessarily unique) element $m={_{A^G}\braket{\xi}{\eta}}$ in $$\M(A)^{G,\alpha}\defeq\left\{m\in \M(A): \alpha_t(m)=m\;\mbox{ for all } t\in G\right\}$$ such that $_{A^G}\braket{\xi}{\eta}\zeta=\int_G\alpha_t(\xi\eta^*)\zeta\,dt$ for all $\zeta\in A_0$; and
\item[{\bf (P3)}] $m\cdot A_0\sbe A_0$ and $A_0\cdot m\sbe A_0$ for $m={_{A^G}\braket{\xi}{\eta}}$ as in (P2).
\end{enumerate}
Given such $A_0$, Rieffel shows in \cite{Rieffel:Proper} that $A^{G,\alpha}:=\cspn\{_{A^G}\braket{\xi}{\eta}: \xi,\eta\in A_0\}\subseteq \M(A)$ is a \cstar{}subalgebra of $\M(A)$ and that $A_0$ equipped with the inner product ${_{A^G}\braket{\xi}{\eta}}$ completes to give a left Hilbert $A^{G,\alpha}$-module ${\F(A_0):=}\overline{A}_0$. The \cstar{}algebra $A^{G,\alpha}$ (or simply $A^G$ if the action $\alpha$ is clear) is called the {\em generalised fixed-point algebra} for the proper action $\alpha$ (with respect to $A_0$). Clearly, if $G$ is compact, this coincides with the classical fixed-point algebra for any dense $A_0\sbe A$. Rieffel also shows that ${\F(A_0)}$ carries a right Hilbert module structure over the reduced crossed product $A\rtimes_{\alpha,r}G$ in such a way that ${\F(A_0)}$ is a Hilbert $A^{G,\alpha}-A\rtimes_{\alpha,r}G$-bimodule.

The module $\F(A_0)$ can be concretely described as the completion of $A_0$ with respect to the Hilbert $A^{G,\alpha}-A\rtimes_{\alpha,r}G$-bimodule structure given by the formulas:
\begin{alignat*}{2}
\bbraket{\xi}{\eta}_{A\rtimes_{\alpha,r}G}(t)&\defeq \Delta(t)^{-1/2}\xi^*\alpha_t(\eta),\quad \xi,\eta\in A_0, t\in G\\
\xi*\varphi &\defeq \int_G\Delta(t)^{-1/2}\alpha_t(\xi\varphi(t^{-1}))dt,\\
_{A^G}\braket{\xi}{\eta}\zeta&\defeq \int_G\alpha_t(\xi\eta^*)\zeta dt,\\
m\cdot \xi&\defeq ma\quad \mbox{(product in $\M(A)$)}.
\end{alignat*}
Here the first formula gives an element in $L^1(G,A)\subseteq A\rtimes_{\alpha,\red}G$ and the second formula
works for all $\varphi\in L^1(G,A)$ for which the integral provides an element in $A_0$.
The bimodule $\F(A_0)$ is always full as a left Hilbert $A^{G,\alpha}$-module, but the right inner product is not full in general
since the ideal $\mathcal I=\cspn\{\bbraket{A_0}{A_0}_{A\rtimes_rG}\}\subseteq A\rtimes_rG$ may be proper. The action $\alpha:G\to \Aut(A)$ is called {\em saturated} (with respect to $A_0$) if this ideal is all of $A\rtimes_{\alpha,r}G$. Then $\F(A_0)$ becomes a Morita equivalence between $A^{G,\alpha}$ and $A\rtimes_{\alpha,r}G$. In general, $\F(A_0)$ becomes a Morita equivalence between $A^{G,\alpha}$ and the ideal
$\mathcal I=\cspn\{\bbraket{A_0}{A_0}_{A\rtimes_rG}\}\subseteq A\rtimes_rG$.

\begin{remark}\label{rem-left-right} We should note that the module $\F(A_0)$ described above is actually the dual of
the $A\rtimes_{\alpha,\red}G-A^{G,\alpha}$ module as constructed originally by Rieffel in \cite{Rieffel:Proper}.
\end{remark}

It is tempting to write $_{A^{G}}\braket{\xi}{\eta}$ as a sort of integral $\int_G\alpha_t(\xi\eta^*)dt$. Although this integral cannot converge
as a Bochner integral in general, one can make sense of it as a \emph{strict-unconditional integral} as defined in \cites{Exel:Unconditional,Exel:SpectralTheory}:

\begin{definition}\label{def-integrable}
One says that a measurable function $f\colon G\to A$ is strictly unconditionally integrable if the net of Bochner integrals $\int_Kf(t)dt$ for $K$ running over all compact subsets of $G$ (ordered by inclusion) converges in the strict topology of $\M(A)$; the strict limit is then denoted by $\int^\su_Gf(t)dt$.

If $(A,\alpha)$ is a $G$\nb-algebra, an element $\xi\in A$ is called \emph{square-integrable} if the function $t\mapsto \alpha_t(\xi\xi^*)$ is strictly unconditionally integrable. We write $A_\si$ for the space of all square-integrable elements of $A$.
The {\em space of integrable elements} $A_\ii$ is then defined as the linear span of $A_\si A_\si^*$, i.e., linear combinations of elements of the form $\xi\eta^*$ with $\xi,\eta\in A_\si$. The $G$\nb-algebra $(A,\alpha)$ is \emph{integrable} if $A_\ii$ (or equivalently $A_\si$) is dense in $A$.
\end{definition}

Rieffel calls integrable actions also ``proper'' in \cite{Rieffel:Integrable_proper}, but as shown in \cite{Meyer:generalised_Fixed} this class of actions is strictly bigger than the Rieffel proper actions of \cite{Rieffel:Proper} as recalled above.

In \cite{Rieffel:Integrable_proper}*{Proposition~4.6} Rieffel shows that (P1) and the density of $A_0$ implies the existence of the strict-unconditional integrals $\int_G^\su\alpha_t(\xi\eta^*)dt$, so that $A_0\sbe A_\si$ and hence $A$ is an integrable $G$\nb-algebra.
More precisely, the proof of Proposition~4.6 in \cite{Rieffel:Integrable_proper} shows the following:

\begin{proposition}
Suppose $\Rel$ is a dense subspace (not necessarily a \Star{}subalgebra) such that
the functions $t\mapsto \xi^*\alpha_t(\eta)$ belong to $L^1(G,A)$ for all $\xi,\eta\in \Rel$.
Then the integrals $\int_G^\su\alpha_t(\xi\eta^*)\,dt$ exist for all $\xi,\eta\in \Rel$.
In particular, if $\Rel\subseteq A$ satisfies Rieffel's condition \textup{(P1)}, it also satisfies \textup{(P2)}.
\end{proposition}

\begin{remark}
(1) The above proposition can be shown by using the following general fact:
an increasing net $(a_i)$ of positive elements in $\M(A)$ converges in the strict topology if and only if
the increasing net $(b^*a_ib)$ converges in the norm of $A$ for all $b$ in a dense subset $D\sbe A$.
This follows from \cite{Kustermans-Vaes:Weight}*{Result~3.4} and the easily verified fact that the set of all $b\in A$ such that $(b^*a_ib)$ converges in norm is closed in $A$. Now, in the setting of the proposition, apply this to the net of integrals
$\int_K\alpha_t(\xi\xi^*)\, dt$ for $K\sbe G$ varying over the compact subsets directed by inclusion and use $D=\Rel$. This gives the existence of $\int^\su_G\alpha_t(\xi\xi^*)\, dt$ for all $\xi\in \Rel$ and the integrals
$\int_G^\su\alpha_t(\xi\eta^*)\, dt$ then exist by the polarization identity.
\medskip
\\
(2) The above proposition has been generalised in \cite{Meyer:generalised_Fixed}*{Proposition~6.5}, where it is shown that the strict-unconditional integrability of $t\mapsto \alpha_t(\xi\eta^*)$ follows from the assumption that the functions $t\mapsto \xi^*\alpha_t(\eta)$ belong to $A\rtimes_{\alpha,r}G$ in a suitable sense (by interpreting these functions as kernels of certain ``Laurent operators'', as explained in \cite{Meyer:generalised_Fixed} or in \cite{Exel:SpectralTheory}; see also \S\ref{sec:counter} below) for $\xi,\eta$ in a dense subspace of $A$.

A similar idea also implies the strict-unconditional integrability of $t\mapsto \alpha_t(\xi\eta^*)$ if the functions $t\mapsto \Delta(t)^{-1/2}\xi^*\alpha_t(\eta)$ belong to $A\rtimes_{\alpha,r}G$ for all $\xi,\eta$ in a dense subspace of $A$.
\end{remark}

The proposition shows that there are some redundancies in Rieffel's original definition of a proper action in \cite{Rieffel:Proper}. Indeed, next we show that only the first condition (P1) is necessary in order to get a Rieffel proper action. For this we first need to fix some notations.

Given any elements $\xi,\eta\in A$, we shall always write $\bbraket{\xi}{\eta}$ for the continuous function $[t\mapsto \Delta(t)^{-1/2}\xi^*\alpha_t(\eta)]$, even if this is not in $L^1(G,A)$ or in $A\rtimes_{\alpha,r}G$. Also, given $\xi\in A$ and
$\varphi:G\to A$ a measurable function, we write $\xi*\varphi\defeq \int_G\Delta(t)^{-1/2}\alpha_t(\xi \varphi(t^{-1}))\, dt$ whenever this makes sense. We shall use the notation:
\begin{align*}
L^1_\Delta(G,A)&\defeq L^1(G,A)\cap \Delta^{1/2}L^1(G,A)\cap \Delta^{-1/2}L^1(G,A)\\
&=\{f\in L^1(G,A):\Delta^{-1/2}f\mbox{ and } \Delta^{1/2}f\in L^1(G,A)\}.
\end{align*}
It is easy to verify the relations:
$$\Delta^{p}(f*g)=(\Delta^{p}f)*(\Delta^{p}g),\quad (\Delta^{p}f)^*=\Delta^{-p}f^*$$
for any power $p\in \R$. In particular it follows that $L^1_\Delta(G,A)$ is a \Star{}subalgebra of $L^1(G,A)$. It is a dense subalgebra because it contains $\contc(G,A)$.

\begin{proposition}\label{prop:CharacterizationRieffelProperActions}
An action $(A,\alpha)$ is Rieffel proper if and only if there is a dense subspace $\Rel\sbe A$ such that for all $\xi,\eta\in \Rel$, we have
$$
\bbraket{\xi}{\eta}=\Big(t\mapsto \Delta(t)^{-1/2} \xi^*\alpha_t(\eta)\Big)\in L^1(G,A)\cap \Delta^{-1/2}L^1(G,A). \leqno{\textup{(P1)}}
$$
In this case, $\bbraket{\Rel}{\Rel}\sbe L^1_\Delta(G,A)$ and $A_0:={\spn}(\Rel*L^1_\Delta(G,A))\cdot (\Rel*L^1_\Delta(G,A))^*$ is a dense \Star{}subalgebra of $A$ which satisfies Rieffel's conditions \textup{(P1)-(P3)}. Hence $(A, \alpha)$ is Rieffel proper with respect to $A_0$.
\end{proposition}

\begin{proof} If the action is Rieffel proper with respect to a dense \Star{}subalgebra $A_0$, we simply take $\Rel=A_0$.
Conversely, given $\Rel$ as in the proposition, we will show that $A_0={\spn}(\Rel*L^1_\Delta(G,A))\cdot (\Rel*L^1_\Delta(G,A))^*$
satisfies Rieffel's conditions (P1)\nb--(P3).  For this we will see that if  $\Rel\sbe A$ is a dense subspace of $A$ satisfying (P1), then $\bbraket{\Rel}{\Rel}\sbe L^1_\Delta(G,A)$.
Moreover, $\tilde\Rel\defeq \spn \Rel*L^1_\Delta(G,A)$ is then automatically a $G$\nb-invariant dense right ideal of $A$ also satisfying
$$\bbraket{\tilde\Rel}{\tilde\Rel}\sbe L^1_\Delta(G,A).$$
Then, $A_0\defeq \spn\tilde\Rel\tilde\Rel^*$ is a $G$\nb-invariant dense \Star{}subalgebra of $A$ contained in $\tilde\Rel$ which satisfies (P1)-(P3).

Given $\xi,\eta\in \Rel$, we have $\bbraket{\xi}{\eta}^*=\bbraket{\eta}{\xi}$ and $(\Delta^{1/2}\bbraket{\xi}{\eta})^*=\Delta^{-1/2}\bbraket{\eta}{\xi}$,
where involution is taken inside $L^1(G,A)$. Therefore,
$\bbraket{\Rel}{\Rel}\sbe L^1(G,A)\cap \Delta^{-1/2}L^1(G,A)$
is equivalent to
$\bbraket{\Rel}{\Rel}\sbe L^1(G,A)\cap \Delta^{1/2}L^1(G,A)$
and hence also equivalent to $\bbraket{\Rel}{\Rel}\sbe L^1_\Delta(G,A)$.

Now, given $f,g\in L^1_\Delta(G,A)$, we have $\bbraket{\xi*f}{\eta*g}=f^**\bbraket{\xi}{\eta}*g$.
Since $L^1_\Delta(G,A)$ is a \Star{}subalgebra of $L^1(G,A)$, it follows that $\tilde\Rel=\Rel*L^1_\Delta(G,A)$ satisfies $\bbraket{\tilde\Rel}{\tilde\Rel}\sbe L^1_\Delta(G,A)$. That $\tilde\Rel$ is a right ideal follows from the identity $(\xi*f)\cdot a=\xi*(f\cdot a)$, where $f\cdot a(t)\defeq f(t)\alpha_t(a)$. And that $\tilde\Rel$ is $G$\nb-invariant follows from $\alpha_t(\xi*f)=\xi*(t\cdot f)$, where $(t\cdot f)(s)\defeq \Delta(t)^{1/2}f(st)$.
It follows that $A_0=\spn\tilde\Rel\tilde\Rel^*\sbe \tilde\Rel$ is a $G$\nb-invariant dense \Star{}subalgebra of $A$.
Since (P1) and (hence also) (P2) hold for $\tilde\Rel$,
they also hold for $A_0\sbe \tilde\Rel$. Finally notice that
$$_{A^G}\braket{\tilde\Rel}{\tilde\Rel}\tilde\Rel=\tilde\Rel*\bbraket{\tilde\Rel}{\tilde\Rel}_{A\rtimes G}\sbe \tilde\Rel*L^1_\Delta(G,A)\sbe \tilde\Rel$$
so that $_{A^G}\braket{A_0}{A_0}A_0\sbe {}_{A^G}\braket{\tilde\Rel}{\tilde\Rel}\tilde\Rel\tilde\Rel^*\sbe \tilde\Rel\tilde\Rel^*\sbe A_0$, i.e., (P3) holds for $A_0$.
\end{proof}

One can also use the subspace $\tilde\Rel$ above, or more generally, any dense subspace $\Rel\sbe A$
satisfying
\begin{equation}\label{eq:Rel-Cont-Inv-Sub}
\bbraket{\Rel}{\Rel}\sbe L^1_\Delta(G,A)\quad\mbox{and }\Rel*L^1_\Delta(G,A)\sbe \Rel
\end{equation}
to build a pre-Hilbert $A^{G,\alpha}-A\rtimes_{\alpha,r}G$-bimodule $\F_0(\Rel)$. More precisely, if $\Rel$ is such a subspace, we can view $\Rel$ as a right pre-Hilbert module over $L^1_\Delta(G,A)\sbe A\rtimes_{\alpha,r}G$
with $(\xi,\eta)\mapsto \bbraket{\xi}{\eta}$ as the inner product and the ``convolution'' $\xi*\varphi=\int_G\Delta(t)^{-1/2}\alpha_t(\xi\varphi(t^{-1}))\, dt$ as the right action. That these are indeed well-defined operations and have the correct properties (in particular, that $\bbraket{\xi}{\xi}$ is positive in $A\rtimes_{\alpha,r}G$) follows from the same arguments as used by Rieffel in \cite{Rieffel:Proper}. Also, the same arguments show that
$$A_0^G\defeq \spn {}_{A^G}\braket{\Rel}{\Rel}=\spn\left\{\int_G^\su\alpha_t(\xi\eta^*)dt:\xi,\eta\in \Rel\right\}$$
is a \Star{}subalgebra of $\M(A)^G$, and that the multiplication in $\M(A)$ and the pairing $(\xi,\eta)\mapsto {}_{A^G}\braket{\xi}{\eta}$ give $\Rel$ the structure of a pre-Hilbert $A^G_0-L^1_\Delta(G,A)$-bimodule.
Defining then $A^G$ to be the closure of $A^G_0$ in $\M(A)$, and completing $\Rel$ with respect to the norm coming from one of the inner products, one gets a Hilbert $A^G-A\rtimes_{\alpha,r}G$-bimodule $\F(\Rel)$. We should also point out that such ideas have been already performed in \cite{Meyer:generalised_Fixed} in a slightly more general context for Hilbert modules and where the inner products $\bbraket{\xi}{\eta}$ do not necessarily lie in $L^1_\Delta(G,A)$, but only in $A\rtimes_{\alpha,r}G$ in general. We refer to \S\ref{sec:counter} for a more detailed discussion of this.

\begin{proposition}\label{prop-uniquemodule}
Let $\Rel\sbe A$ be a dense subspace satisfying~\eqref{eq:Rel-Cont-Inv-Sub}, and define $\tilde\Rel\defeq \spn\Rel*L^1_\Delta(G,A)\sbe \Rel$ and $A_0\defeq \spn\tilde\Rel\tilde\Rel^*$. Then all Hilbert bimodules $\F(\Rel)$, $\F(\tilde\Rel)$ and $\F(A_0)$ are canonically isomorphic.
\end{proposition}
\begin{proof}
The inclusions $A_0\sbe \tilde\Rel\sbe \Rel$ extend to embeddings $\F(A_0)\into \F(\tilde\Rel)\into \F(\Rel)$ of Hilbert $A\rtimes_{\alpha,r}G$-modules. To show that these embeddings are isomorphisms, it is enough (by the Rieffel Correspondence Theorem \cite{Raeburn-Williams:Morita_equivalence}*{Corollary~3.33}) to show that the ideals
$$\cspn\bbraket{A_0}{A_0}\sbe \cspn\bbraket{\tilde\Rel}{\tilde\Rel}\sbe \cspn\bbraket{\Rel}{\Rel}$$
of $A\rtimes_{\alpha,r}G$ coincide.
To see this  first observe from  a simple computation that
\begin{equation}\label{eq-formulas}
\bbraket{\xi*f}{\eta*g}=f^**\bbraket{\xi}{\eta}*g\quad\text{and}\quad \bbraket{\xi a}{\eta b}=i_A(a)^*\braket{\xi}{\eta}i_A(b)
\end{equation}
for all $\xi,\eta\in \Rel$, $f,g\in \contc(G,A)$, and  $a,b\in A$, where $\big(i_A(a)f\big)(t)\defeq af(t)$ and $\big(fi_A(a)\big)(t)\defeq f(t)\alpha_t(a)$.
Note that the latter formulas determine the canonical inclusion $i_A:A\to \M(A\rtimes_{\alpha,\red}G)$.
Since $\tilde\Rel^*$ is a dense left ideal of $A$, there exists a bounded right approximate unit $(e_i)$ of $A$ in $\tilde\Rel^*$.
 It follows then from the second equation in (\ref{eq-formulas}) that
$$\bbraket{\xi e_i}{\eta e_i}=i_A(e_i^*)\bbraket{\xi}{\eta} i_A(e_i)\to \bbraket{\xi}{\eta},$$
for all $\xi,\eta\in \tilde\Rel$, which shows that $\cspn\bbraket{A_0}{A_0}= \cspn\bbraket{\tilde\Rel}{\tilde\Rel}$.
Choosing a self-adjoint approximate unit $(\varphi_i)$ of $A\rtimes_{\alpha,\red}G$ in $C_c(G,A)\subseteq L^1_\Delta(G,A)$,
we get
$$\bbraket{\xi*\varphi_i}{\eta*\varphi_i}=\varphi_i*\bbraket{\xi}{\eta}*\varphi_i\to \bbraket{\xi}{\eta}$$
for all $\xi,\eta\in \Rel$, which proves $\cspn\bbraket{\tilde\Rel}{\tilde\Rel}= \cspn\bbraket{\Rel}{\Rel}$.
\end{proof}

\begin{remark}\label{rem-fix}
Let $(A,\alpha)$ be a $G$\nb-algebra and suppose that $A_0 \subseteq A$ is a $G$\nb-invariant dense \Star{}subalgebra of $A$
which satisfies Rieffel's conditions \textup{(P1) -- (P3)}. It follows not necessarily from these conditions that
$A_0*L^1_\Delta(G,A)\subseteq A_0$, hence it is not necessarily true that $A_0$
satisfies~\eqref{eq:Rel-Cont-Inv-Sub}. However, it is clear that $\tilde\Rel:= A_0*L^1_\Delta(G,A)$ must be contained in
the intersection of Rieffel's module $\F(A_0)=\overline{A_0}$ with $A$. Using approximate units $(\varphi_i)$ in $C_c(G,A)$ as
above, it follows that we get $\F(A_0)=\F(\tilde\Rel)=\F(\tilde\Rel \tilde\Rel^*)$, so we see that applying our procedure to
$\Rel:=A_0$ leads to the original  Hilbert bimodule $\F(A_0)$ of Rieffel's and the corresponding fixed-point algebra.
\end{remark}

The following facts,  which apparently have not been noticed before in the literature, are now easy consequences of our characterisation of Rieffel proper actions and will be used frequently in this paper. For notation, if $(A,\alpha)$ is a $G$\nb-algebra and if $\Rel\subseteq A$
is as in  Proposition \ref{prop:CharacterizationRieffelProperActions}, then we shall say that $\alpha:G\to \Aut(A)$
is {\em Rieffel proper with respect to $\Rel\subseteq A$.}

\begin{corollary}\label{cor-functorial}
Suppose that $\alpha:G\to\Aut(A)$ is Rieffel proper with respect to the dense subspace $\Rel_A\subseteq A$.
Let $(B,\beta)$ be another $G$\nb-algebra and let $\Phi:A\to \M(B)$ be a nondegenerate $G$\nb-equivariant \Star{}homomorphism.
Then $\beta:G\to \Aut(B)$ is  Rieffel proper with respect $\Rel_B:=\Phi(\Rel_A)B\subseteq B$.
Moreover, there is a canonical isomorphism of Hilbert $B\rtimes_{\beta,r} G$-modules:
$$\F(\Rel_B)\cong \F(\Rel_A)\otimes_{\Phi\rtimes_{r} G}B\rtimes_{\beta,r} G,$$
where $\Phi\rtimes_r G\colon A\rtimes_{\alpha,r}G\to \M(B\rtimes_{\beta, r}G)$ denotes the (nondegenerate) \Star{}homomorphism associated to $\Phi$.
In particular, if $A$ is saturated with respect to $\Rel_A$, then $B$ is saturated with respect to $\Rel_B$.
\end{corollary}
\begin{proof} Since $\Phi$ is nondegenerate, $\Rel_B$ is dense in $B$.
Let $x,y\in \Rel_A$, $b,c\in B$ and let $\xi=\Phi(x)b$ and $\eta=\Phi(y)c$.
Then $s\mapsto \bbraket{x}{y}(s)\in L^1_\Delta(G,A)$, and hence we
get $s\mapsto \bbraket{\xi}{\eta}(s)=b\Phi(\bbraket{x}{y}(s))\beta_s(c)\in L^1_\Delta(G,B)$.
The final assertion is a particular case of Corollary~7.1 in \cite{Meyer:generalised_Fixed} by using the canonical isomorphism $B\cong A\otimes_\Phi B$ as $G$\nb-equivariant Hilbert $B$-modules and the fact that Rieffel proper actions are proper in the sense of Exel-Meyer, as we shall explain in
\S\ref{sec:counter} below.
\end{proof}

\begin{corollary}\label{cor-tensor}
Suppose that $\alpha:G\to \Aut(A)$ and $\beta:H\to \Aut(B)$ are Rieffel proper with respect to $\Rel_A\subseteq A$ and $\Rel_B\subseteq B$, respectively.
Then $\alpha\otimes\beta: G\times H\to \Aut(A\otimes_\nu B)$ is proper with respect to  $\Rel_A\odot\Rel_B\subseteq A\otimes_\nu B$, where ``$\otimes_\nu$'' denotes either the maximal or the minimal tensor product of $A$ with $B$.
\end{corollary}
\begin{proof} If $\xi_1,\xi_2\in \Rel_A$ and $\eta_1,\eta_2\in \Rel_B$, then
$$\bbraket{\xi_1\otimes\eta_1}{\xi_2\otimes \eta_2}(s,h)=\bbraket{\xi_1}{\xi_2}(s)\cdot\bbraket{\eta_1}{\eta_2}(h)$$
lies in $L^1_\Delta(G\times H, A\otimes_\nu B)$.
\end{proof}

\begin{corollary}\label{cor-lift}
 Suppose that $N$ is a compact normal subgroup of $G$ and let $\tilde\alpha: G/N\to \Aut(A)$ be Rieffel proper
with respect to $\Rel_A\subseteq A$.
Then the inflated action $\alpha:=\inf\tilde\alpha:G\to \Aut(A)$ is also Rieffel proper with respect to
$\Rel_A$.
\end{corollary}
\begin{proof} This follows immediately from the fact that in this situation a function  $f: G/N\to A$ is integrable if and
only if the function $f\circ q$ is integrable over $G$, when $q:G\to G/N$ denotes the quotient map.
\end{proof}

\section{Properness of dual actions on twisted group algebras}\label{sec-twisted}

Recall that for any {\bf abelian} locally compact group $G$ and any Borel 2-cocycle $\om \in Z^2(G,\T)$ the twisted group algebra
$C^*(G,\om)$ is a \cstar{}completion of the Banach algebra $L^1(G,\om)$ consisting of all $L^1$-functions on
$G$ with convolution and involution  twisted by $\om$ as follows:
$$ f*_\om g(t)=\int_G f(s) g(t-s)\om(s, t-s)\, ds\quad \text{and}\quad f^*(s)=\overline{\om(s, -s) f(-s)}.$$
The same convolution formula defines a  \Star{}representation
 $$\lambda_\om: L^1(G,\om)\to \mathcal B(L^2(G)); \quad \lambda_\om(f)\xi=f*_\om\xi$$
(which makes  sense  for $\xi\in L^1(G)\cap L^2(G)$) and  then
$$C^*(G,\om):=\overline{\lambda_\om(L^1(G,\om))}\subseteq \mathcal B(L^2(G)).$$
There is a canonical dual action $\widehat{\om}:\widehat{G}\to\Aut(C^*(G,\om))$ of the dual group $\widehat{G}$
on $C^*(G,\om)$ which is given on the dense subalgebra $L^1(G,\om)$ by
$$\widehat{\om}_{\chi}(f)=\chi\cdot f,\quad \forall \chi\in \widehat{G}, f\in L^1(G,\om).$$
We want to show in this section that this action is always Rieffel proper.

For a  detailed study of twisted group algebras of abelian groups we refer to the
paper \cite{Bag-Klep}.  Note that two cocycles $\om$ and $\om'$ are called {\em similar} (or cohomologous) if
there exists a Borel function $c:G\to\T$ such that $\om'(s,t)=c(s)c(t)\overline{c(st)}\om(s,t)$ for all $s,t\in G$.
In this case the mapping $f\mapsto c\cdot f$ (pointwise multiplication) on $L^1$-functions extends to
an isomorphism $C^*(G,\om)\cong C^*(G,\om')$ which commutes with the dual actions.
Thus for our purposes it is enough to fix any representative of $\om$ under similarity, or, equivalently,
to look at classes in $H^2(G,\T)=Z^2(G,\T)/\sim$.

If $\om\in Z^2(G,\T)$, then there is a continuous homomorphism $h_\om:G\to \widehat{G}$ given by
$$h_\om(s)(t)=\om(s,t)\overline{\om(t,s)}.$$
It is shown in \cite{Bag-Klep} that two cocycles $\om$ and $\om'$ on the abelian group $G$ are similar
if and only if $h_\om=h_{\om'}$. Moreover, if
$$S_\om:=\ker h_\om=\{ s\in G: \om(s,t)=\om(t,s)\;\forall t\in G\}$$
denotes the {\em symmetrizer group} of $\om$, then Baggett and Kleppner show in \cite{Bag-Klep}*{Theorem 3.1}
that $\om$ is always similar
to a cocycle inflated  from some cocycle $\tilde\om$ on $G/S_\om$.

In what follows, we shall need the following basic fact on computing the cohomology group $H^2(G,\T)=Z^2(G,\T)/\sim$
for abelian groups $G$, which is actually a very special case of Mackey's formula \cite{Mackey:UnitaryI}*{Proposition 9.6}
together with \cite{Kleppner:Multipliers}*{Propositions 1.4 and 1.6}.

\begin{lemma}\label{lem-1} Suppose $G=H\times N$ is the direct product of the abelian locally compact groups $H$ and $N$
such that one of them is locally euclidean.
Then every cocycle $\om \in Z^2(G,\T)$ is similar to a product $\om_H\cdot\om_\eta\cdot \om_N$, where $\om_H$ and $\om_N$
are the restrictions of $\om $ to $H$ and $N$, respectively, and $\eta: H\times N\to \T$ is a continuous bi-character
and $\om_{\eta}\big((h_1, n_1), (h_2, n_2)\big)=\eta(h_1, n_2)$.
\end{lemma}

We also need the following well-known fact:

\begin{lemma}\label{lem-2} Suppose that $K$ is a compact abelian group. Then every cocycle $\om\in Z^2(K,\T)$
is similar to a cocycle inflated from a finite quotient $K/N$ of $K$.
\end{lemma}
\begin{proof} Just observe that the kernel $S_\om$ of $h_\om: K\to \widehat{K}$ is open in $K$ and hence has
finite index in $K$. The result then follows from the above discussions.
\end{proof}

It follows from \cite{Bag-Klep}*{p.314} that every cocycle on $\R^n$ is similar to one of the form
$\om_J(s,t)=e^{2\pi i\langle Js,t\rangle}$ for a unique skew-symmetric matrix $J\in \Mat_n(\R)$ and
 that $C^*(\R^n,\om_J)$ is commutative if and only if
$J=0$.
In this case it follows from an easy exercise on Fourier transforms that
$$C^*(\R^n,\om_J)= C^*(\R^n)\cong C_0(\widehat{\R^n})\cong C_0(\R^n)$$
 and the dual
action of $\R^n\cong \widehat{\R^n}$ on $C^*(\R^n)$  is transformed to the translation action on $C_0(\R^n)$.
It is clear that this action is  proper in the strongest sense.

In order to show that the dual actions of $\R^n$ on $C^*(\R^n,\om_J)$ are Rieffel proper for all skew-symmetric $J$,
we want to rely on Rieffel's study of his $J$-deformed algebras $C_0(\R^n)_J$ of $C_0(\R^n)$.
For this let  $\Sch(\R^n)$ be the Schwartz space
of rapidly decreasing functions on $\R^n$ equipped with the translation action $\tau$  of $\R^n$.
In \cite{Rieffel:Deformation, Rief-K} Rieffel considers the deformed multiplication $\times_J$ on $\Sch(\R^n)$ given by
the formula
$$f\times_J g(x)=\int_{\R^n}\int_{\R^n} f(x-Ju) g(x-v) e^{2\pi i\langle u,v\rangle}\,dv\,du.$$
We should note that in general this double integral only exists in the given order, and that
Fubini's theorem cannot be applied to it!
We write $\Sch(\R^n)_J$ for $\Sch(\R^n)$ equipped with this multiplication.
Rieffel shows that together with the involution  $f^*(x):=\overline{f(x)}$, $\Sch(\R^n)_J$ becomes a \Star{}algebra
and there is a canonical faithful \Star{}representation
$L_J: \Sch(\R^n)_J\to \mathcal B(L^2(\R^n))$ by bounded operators
given by the formula $L_J(f)\xi=f\times_J\xi$ for $\xi\in \Sch(\R^n)\subseteq L^2(\R^n)$.  The completion $C_0(\R^n)_J:=\overline{\Sch(\R^n)_J}$ with respect to the operator norm of $\mathcal B(L^2(G))$ is the $J$-deformation of $C_0(\R^n)$.

Rieffel shows in \cite{Rief-K}*{\S 2} that the
translation action $\tau$ of $\R^n$ on $\Sch(\R^n)_J$ extends to an action (denoted $\tau_J$)
 of $\R^n$ on $C_0(\R^n)_J$ which is Rieffel proper with respect to the dense subalgebra $\Sch(\R^n)_J$.
 The following lemma extends the above isomorphism $C^*(\R^n)\cong C_0(\R^n)$ to the case of non-zero $J$.
The lemma must be known by Rieffel (see \cite{Rief-Q}*{p. 70}), but we did not find an explicit proof
 in the literature. In what follows we write $\Sch(\R^n,\om_J)$ for the dense subalgebra of $L^1(\R^n,\om_J)$
 consisting of Schwartz functions.

 \begin{lemma}\label{lem-iso}
For every skew-symmetric matrix $J$ the
 Fourier transform
 $$\mathcal F:\Sch(\R^n)_J \to\Sch(\R^n,\om_J),\quad  f\mapsto\widehat{f}$$
  extends to a $\tau_J-\widehat{\om_J}$-equivariant isomorphism $\Phi: C_0(\R^n)_J\to C^*(\R^n,\om_J)$.
 \end{lemma}
 \begin{proof}
Recall that $\widehat{f}(u)=\int_{\R^n} f(s) e^{-2\pi i\langle s,u\rangle}\,ds$. Then, for
 $f,g\in \Sch(\R^n)$ one checks the formula $\widehat{f\times_J g}=\widehat{f}*_\om \widehat{g}$
(mentioned by Rieffel in \cite{Rief-Q}*{p. 70}) as follows:
First observe that
\begin{align*}
f\times_J g(x) &=\int_{\R^n}\int_{\R^n} f(x- Ju) g(x-v)e^{2\pi i\langle v,u\rangle}\,dv\,du\\
&\stackrel{v\mapsto x-v}{=} \int_{\R^n}\int_{\R^n} f(x- Ju) g(v)e^{2\pi i\langle x-v,u\rangle}\,dv\,du\\
&= \int_{\R^n} f(x-Ju) \widehat{g}(u) e^{2\pi i \langle x,u\rangle}\,du.
\end{align*}
Applying the Fourier transform to this expression gives
\begin{align*}
\widehat{f\times_J g}(z)&=
\int_{\R^n}\int_{\R^n}  f(x-Ju) \widehat{g}(u) e^{2\pi i \langle x,u-z\rangle}\,du\,dx\\
&\stackrel{x\mapsto x+Ju}{=}
\int_{\R^n}\int_{\R^n}  f(x) \widehat{g}(u) e^{-2\pi i \langle x+Ju, z-u\rangle}\,dx\,du\\
&=
\int_{\R^n} \widehat{f}(z-u)  \widehat{g}(u) e^{-2\pi i \langle Ju, z-u\rangle}\,du\\
&\stackrel{u\mapsto z-u}{=}
\int_{\R^n} \widehat{f}(u)  \widehat{g}(z-u) e^{2\pi i \langle Ju, z-u\rangle}\,du\\
&= \widehat{f}*_{\om_J}\widehat{g}(z),
\end{align*}
 where in the second to last equation we  used  $J^t=-J$. The formula also shows that
the Plancherel isomorphism $\mathcal F: L^2(\R^n)\to L^2(\R^n)$ induces a unitary equivalence between $\lambda_\om\circ \Phi$ and
$L_J$, which implies that $\Phi$ is isometric with respect to the \cstar{}norms.
\end{proof}

As mentioned before, the above lemma combined with the results in \cite{Rief-K}*{\S 2} yields the following:
\begin{corollary}
The dual actions of $\R^n$ on $C^*(\R^n,\om_J)$ are always Rieffel proper with respect to $\Sch(\R^n,\om_J)$.
\end{corollary}

We now use this fact to prove:

\begin{theorem}\label{thm-Rieffel proper}
Let $G$ be any second countable abelian group and let $\om\in Z^2(G,\T)$. Then the dual action
$\widehat{\om}:G\to \Aut(C^*(G,\om))$ is Rieffel proper and saturated.
\end{theorem}
\begin{proof}
Recall first that the structure theorem for abelian locally compact groups says that
$G$ is isomorphic to a direct product $\R^n\times H$ for some $n\in \N_0$ such that $H$
contains a compact open subgroup $K$ (\eg see \cite{DE}*{Theorem 4.2.1}).
Let $\tilde\om$ denote the restriction of $\om$ to
$\R^n\times K$. There are no non-trivial  bicharacters  $\eta:\R^n\times K\to \T$, since any such
$\eta$ would induce a nontrivial homomorphism from $K$ into $\widehat{\R^n}\cong \R^n$.
It follows from Lemma \ref{lem-1} that $\tilde\om$ is similar to the product $\om_{\R^n}\cdot\om_K$, where $\om_{\R^n}$ and $\om_K$
denote the restrictions of $\om$ to $\R^n$ and $K$, respectively. It follows then from Lemma \ref{lem-2}
that, after passing to a finite-index subgroup $\tilde{K}\subseteq K$, if necessary, we may
assume without loss of generality, that $\om_K$ is similar to $1_K$. This implies that
$C^*(\R^n\times K, \tilde\om)\cong C^*(\R^n,\om_{\R^n})\otimes C^*(K)\cong C^*(\R^n,\om_{\R^n})\otimes C_0(\widehat{K})$.
Since the dual action of $\widehat{\R^n}\cong \R^n$ on $C^*(\R^n,\om_{{\R^n}})$ is proper with respect to $\Sch(\R^n,\om_{\R^n})$
and the action of $\widehat{K}$ on $C_0(\widehat{K})$ is proper with respect to $C_c(\widehat{K})$,
it follows from Corollary~\ref{cor-tensor} that the dual action of $\widehat{\R^n\times K}\cong \R^n\times \widehat{K}$ on
$C^*(\R^n\times K, \tilde\om)$ is proper with respect to $D_0:=\Sch(\R^n,\om_{{\R^n}})\odot \F_K^{-1}(C_c(\widehat{K}))$,
where $\F_K:C^*(K)\to C_0(\widehat{K})$ denotes Fourier transform for $K$.

Since $L\defeq \R^n\times K$ is an open subgroup of $G$, the dual group $\widehat{L}$ is a quotient of
$\widehat{G}$ with respect to the compact (normal) subgroup $N:=\widehat{G/L}$. Hence it
follows from Corollary~\ref{cor-lift} that the inflation of the dual action of $\widehat{\R^n\times K}$ on $C^*(\R^n\times K,\tilde\om)$ to $\widehat{G}$
is also Rieffel proper.
Now, the restriction of the $\om$-regular representation $\lambda_\om: G\to \U(L^2(G))$ to $\R^n\times K$
induces a $\widehat{G}$-equivariant nondegenerate \Star{}homomorphism, and hence the dual action
of $\widehat{G}$ on $C^*(G,\om)$ is Rieffel proper with respect to
$C^*(G,\om)_0:=\lambda_\om(D_0) C^*(G,\om)$ by Corollary~\ref{cor-functorial}.

Finally, to see that the dual action is also saturated, we simply note that $C^*(G,\om)\rtimes_{\widehat{\om}}\widehat{G}$
is isomorphic to the compact operators on $L^2(G)$, which follows from general Takesaki-Takai duality for crossed products by
twisted actions \cite{Quigg:Duality_twisted}. Since $\mathcal K(L^2(G))$ is simple, the $C^*(G,\om)\rtimes_{\widehat{\om}}\widehat{G}$-valued
inner product on $\F(C^*(G,\om)_0)$ must be full.
\end{proof}

Given any action $\alpha:G\to \Aut(A)$ of the abelian locally compact group $G$ on the $C^*$-algebra $A$ together with
a $2$-cocycle $\om \in Z^2(G,\T)$, we can form the Busby-Smith twisted action $(\alpha, \om\cdot 1_A)$ of $G$ on $A$ (we refer to \cite{Packer-Raeburn:Stabilisation} for details on crossed products by Busby-Smith
twisted actions). It is then easily seen that the canonical embedding of $G$ into $\M(A\rtimes_{\alpha,\om}G)$
is an $\om$-representation, and therefore integrates to give a $\widehat{G}$-equivariant \Star{}homomorphism
$\phi: C^*(G,\om)\to{\M}(A\rtimes_{\alpha,\om}G)$. Thus, as a direct consequence of
Corollary~\ref{cor-functorial}, we get:

\begin{corollary}\label{cor:DualActionsTwistedAreRieffelProper}
Let $(A,\alpha)$ be any action of a second countable locally compact abelian group $G$ on a \cstar{}algebra $A$ and let $\omega$ be any Borel $2$-cocycle on $G$. Then the dual action $\dual\alpha_\omega$ of $\dualG$ on the twisted crossed product $A\rtimes_{\alpha,\omega}G$ is Rieffel proper and saturated.
\end{corollary}

Recall that an action $\alpha:G\to\Aut(A)$ is called {\em weakly proper} if there exists a locally compact proper $G$\nb-space
$X$ together with a nondegenerate \Star{}homomorphism $\phi:C_0(X)\to \M(A)$. It is well known that every weakly proper
action is Rieffel proper. This has first been observed by Rieffel in \cite{Rieffel:Integrable_proper}, but
follows also easily from Corollary \ref{cor-functorial} since for proper $G$\nb-spaces $X$ the
corresponding action of $G$ on $C_0(X)$ is Rieffel proper with respect to $A_0=C_c(X)$.
As mentioned in the introduction,
weakly proper actions enjoy a number of nice properties which are unknown for general Rieffel proper actions.
On the other hand, so far it seemed to be open (at least to us), whether the class of Rieffel proper actions really differs from the
class of weakly proper actions, since most standard examples of Rieffel proper actions given in the literature,
like dual actions on {\em ordinary} crossed products,
are also weakly proper. In what follows next, we show that dual actions on twisted group algebras $C^*(G,\om)$
for connected abelian groups $G$ are weakly proper if and only if the cocycle $\om$ is similar to the trivial
cocycle $1_G$, hence providing many examples of Rieffel proper actions which are not weakly proper.

We should note that the notion ``weakly proper'' has been introduced by the authors in \cite{Buss-Echterhoff:Exotic_GFPA}
in order to differentiate them from proper actions in the (very strong) sense of Kasparov in which we have a proper $G$\nb-space $X$
together with a nondegenerate \Star{}homomorphism $\phi:C_0(X)\to\mathcal Z\M(A)$, where $\mathcal Z\M(A)$ denotes the {\em center} of the
multiplier algebra $\M(A)$.

We need the following lemma, in which $\M(C^*(G,\om))^{\widehat{G}}$ denotes the {\em classical} fixed-point algebra:
$$\M(C^*(G,\om))^{\widehat{G}}=\{m\in \M(C^*(G,\om)): \widehat{\om}_{\chi}(m)=m
\;\forall \chi\in \widehat{G}\}.$$

\begin{lemma}\label{lem-fix}
For any abelian locally compact group we have
$\M(C^*(G,\om))^{\widehat{G}}=\C$.
\end{lemma}
Indeed, this lemma is a special case of a much more general result for locally compact quantum groups given by
Vainerman and Vaes in \cite{Vaes-Vainerman:Extension_of_lcqg}*{Theorem 1.11}. For readers which are not familiar with quantum group techniques we
 present a direct proof here:

\begin{proof}[Proof of Lemma \ref{lem-fix}]
Consider the regular representation $\lambda_\om: C^*(G,\om)\to \mathcal B(L^2(G))$ as introduced above.
Since this is a nondegenerate representation, it extends to  a faithful representation of $\M(C^*(G,\om))$
on $\mathcal B(L^2(G))$ which is then contained in the double commutant  $\vN(C^*(G,\om))=\lambda_\om(C^*(G,\om))''\subseteq \mathcal B(L^2(G))$.
Let us define the {\em right regular} $\om$-representation $\rho_\om:G\to \U(L^2(G))$ by
$$\big(\rho_\om(s)\xi\big)(r)= \om(r-s, s)\xi(r-s).$$
Then a short (but tricky) computation using the cocycle identity $\om(s,t)\om(s+t,r)=\om(s, t+r)\om(t,r)$ for $s,t,r\in G$
shows that $\rho_\om(s)$ commutes with $\lambda_{\om}(t)$ (given by the formula $\big(\lambda_\om(t)\xi\big)(r)=\om(t, r-t)\xi(r-t)$)
for all $t\in G$ and hence with $\lambda_\om(f)=\int_G f(t)\lambda_{\om}(t)\,dt$ for all $f\in L^1(G,\om)$.
This shows that $\rho_\om(s)\in \vN(C^*(G,\om))'$.

Recall now that the dual action  $\widehat{\om}:\widehat{G}\to \Aut(C^*(G,\om))$ is given by
$\widehat{\om}_\chi(f)=\chi\cdot f$ for $\chi\in \widehat{G}, f\in L^1(G,\om)$. Let $U: \widehat{G}\to \U(L^2(G))$ be the unitary representation
given by $U_\chi\xi=\chi\cdot\xi$. Then a short computation shows that
$$\lambda_{\om}(\widehat{\om}_\chi(f))=U_\chi \lambda_{\om}(f) U_{\bar\chi}$$
and hence the action extends uniquely to an action on $\vN(C^*(G,\om))$ via
$\widehat{\om}_\chi(T):=U_\chi TU_{\bar\chi}$.
If $T\in \vN(C^*(G,\om))$ is invariant under this action, it follows that $T$ commutes with $U_\chi$ for all $\chi\in \widehat{G}$.
Taking the integrated form, it follows that $T$ commutes with $U(C^*(\widehat{G}))\subseteq \mathcal B(L^2(G))$.
But if we identify $C^*(\widehat{G})$ with $C_0(G)$ by Gelfand transformation and Pontrjagin duality, a short computation shows
that $U(C^*(\widehat{G}))=M(C_0(G))$, where $M:C_0(G)\to\mathcal B(L^2(G))$ denotes the representation by
multiplication operators.

Hence we conclude that every $T$ in the fixed-point algebra $\vN(C^*(G,\om))^{\widehat{G}}$ must commute with
$M(C_0(G))\cup \rho_\om(G)$. Now define a new cocycle $\tilde\om\in Z^2(G,\T)$ by
$\tilde\om(s,t)=\om(t,s)$. The cocycle identity for $\om$ directly translates
into the cocycle identity for $\tilde\om$ and one easily checks that $\rho_\om=\lambda_{\tilde\om}$.
Consider the twisted dynamical system $(C_0(G), G, \tau, \tilde{\om})$  in the sense of
Busby \& Smith (e.g. see \cite{Packer-Raeburn:Stabilisation}). One then checks that $(M, \rho_\om)$ is a covariant representation of
this system on $L^2(G)$ whose integrated form maps $C_c(G\times G)\subseteq L^1(G, C_0(G))$ onto a dense
subspace of $\mathcal K(L^2(G))$. Since $T$ commutes with $(M, \rho_\om)$, it also commutes
with $\mathcal K(L^2(G))$, which implies that $T\in \C 1$.
\end{proof}

Since a generalised fixed-point algebra $A^G$ for a $G$\nb-algebra $A$, if exists, must lie in the classical
fixed-point algebra $\M(A)^G$ of $\M(A)$, we directly get:

\begin{corollary}\label{cor-fixed}
The generalised fixed-point algebra $C^*(G,\om)^{\widehat{G}}$ with respect to any dense subspace $\mathcal R\subseteq C^*(G,\om)$
which implements Rieffel properness of the dual action of $\widehat{G}$ on $C^*(G,\om)$ as in
Proposition \ref{prop:CharacterizationRieffelProperActions} is isomorphic to $\C$.
\end{corollary}

We use the above result to show:

 \begin{proposition}\label{prop-main}
 Let $G$ be a connected abelian group and let $\om\in Z^2(G,\T)$. Then the  dual action
of $\widehat{G}$ on $C^*(G,\om)$ is weakly proper if and only if $\om$ is similar
 to the trivial cocycle $1_G$.
 \end{proposition}
\begin{proof}
If $\om$ is trivial, then $C^*(G,\om)\cong C_0(\widehat{G})$ with the usual translation action of $\widehat{G}$, which is weakly proper.

Suppose now that $\om$ is nontrivial. By the structure theorem for locally compact abelian groups
we have $G\cong\R^n\times K$ for some connected compact group $K$. As in the proof of
Theorem \ref{thm-Rieffel proper} we can argue that $\om$ is similar to a cocycle $\om_{\R^n}\cdot 1_K$,
and hence $C^*(G,\om)\cong C^*(\R^n,\om_{\R^n})\otimes C_0(\widehat{K})$.
Assume now that there exists a $\widehat{G}$-equivariant nondegenerate \Star{}homomorphism
$\phi:C_0(X)\to \M(C^*(G,\om))$. Then the restriction of the dual action to the factor $\widehat{\R^n}$ in
$\widehat{G}=\widehat{\R^n}\times \widehat{K}$ induces the structure a weakly proper action
of $\R^n\cong \widehat{\R^n}$ on $C^*(G,\om)\cong C^*(\R^n,\om_{\R^n})\otimes C_0(\widehat{K})$ such that the action on the
second factor is trivial. Evaluation of the second factor at the trivial character $1_K\in \widehat{K}$
induces an $\R^n$-equivariant quotient map $C^*(G,\om)\to C^*(\R^n,\om_{\R^n})$, which then induces
the structure  $\varphi:C_0(X)\to \M(C^*(\R^n,\om))$ of a weakly proper action of $\R^n$ on $C^*(\R^n,\om_{\R^n})$.

We need to show that this is impossible.
For this observe that $\R^n$ equipped with the translation action
 of $\R^n$ is a model for the
 universal proper $\R^n$-space, i.e., if $X$ is any proper $\R^n$-space, then there exists an $\R^n$-equivariant
 continuous map $\psi:X\to \R^n$. This  follows from the well know fact that any proper $\R^n$-space
 is a principal (i.e., locally trivial) $\R^n$-bundle (e.g., by Palais's slice theorem), and that any principal $\R^n$-bundle is trivial
 by contractibility of $\R^n$ (e.g., see \cite{Huse}).  Hence we obtain a nondegenerate $\R^n$-map $\psi^*: C_0(\R^n)\to C_b(X)=\M(C_0(X))$
 by $\psi^*(f)=f\circ \psi$ and composing this with $\varphi$ we may assume without loss of generality that
 $X=\R^n$ (see also \cite{Buss-Echterhoff:Imprimitivity}*{Remark~3.13(d)}).

 Assuming this we see that our assumption implies that $C^*(\R^n,\om)$ is  a weakly proper $\R^n\rtimes \R^n$-algebra
 and hence it follows from Landstad duality for coactions of $\R^n$ (see \cite{Quigg}*{Theorem 3.3} or
 \cite{Buss-Echterhoff:Exotic_GFPA}), which in the present case is just Landstad duality for
 actions of $\widehat{\R^n}\cong \R^n$, that there exists an action $\alpha$ of $\R^n$ on the generalised fixed-point algebra
 $C^*(\R^n,\om)^{\R^n}$  such
 that $C^*(\R^n,\om)\cong C^*(\R^n,\om)^{\R^n}\rtimes_\alpha\R^n$. But it follows from the
 construction of this fixed-point algebra (e.g., see \cite{Buss-Echterhoff:Exotic_GFPA}) that it must
 lie in the classical fixed-point algebra $\M(C^*(\R^n,\om))^{\R^n}$ which is $\C$ by Lemma \ref{lem-fix}.
  But the only action on $\C$ is the trivial one, and we conclude that
$C^*(\R^n,\om)\cong \C\rtimes \R^n\cong C^*(\R^n)$ is commutative, which contradicts
 the fact that $\om$ is nontrivial.
\end{proof}

Since actions of compact groups $K$ are always proper in any given sense (they are always Kasparov proper for the proper
$K$-space $\{\pt\}$), it is clear that for discrete groups $G$ the dual actions of $\widehat{G}$ on $C^*(G,\om)$
are always weakly proper. Indeed, this observation can be extended as follows:

\begin{remark}
 Let $G=\R^n\times H$ be any abelian locally compact  group such that the restriction of $\om$
 to $\R^n$ is trivial. Then as in the proof of Theorem \ref{thm-Rieffel proper} we may argue that
 there exists an open subgroup $M=\R^n\times K$ such that the restriction of $\om$ to
 $M$ is trivial. Then the dual action of $\widehat{M}$ on $C^*(M,\om)\cong C_0(\widehat{M})$ is
 weakly proper (with $X=\widehat{M}$). Using the fact that $\widehat{M}=\widehat{G}/N$
 for the compact subgroup $N:=\widehat{G/M}$, we see that the
 translation action of $\widehat{G}$ on $\widehat{M}$ is proper, too.
Hence, the restriction of $\lambda_\om$ to $M$ provides a $\widehat{G}$-equivariant
imbedding of $C_0(\widehat{M})\cong C^*(M,\om)$ into $C^*(G,\om)$, which proves that
the dual action of $\widehat{G}$  on $C^*(G,\om)$ is weakly proper.

Of course one might wonder, whether the converse of this observation is true: Is
the dual action of $\widehat{G}$ on $C^*(G,\om)$ weakly proper if and only if
the restriction of $\om$ to $\R^n$ is trivial?

Note that there exist quite interesting twisted group algebras given by such examples.
For instance, let $\om$ be the Heisenberg-cocycle on $\R\times\Q$ given by the formula
$$\om\big((s,q), (t, r)\big)=e^{2\pi i sr}.$$
Then the twisted group algebra $C^*(\R\times\Q, \om)$ is isomorphic to the crossed product
$C_0(\R)\rtimes \Q$ where $\Q$ acts by translation on $\R$. Since this action is minimal
(i.e., all orbits are dense), this algebra is simple.
\end{remark}

\section{Exel-Meyer proper actions and counterexamples}\label{sec:counter}

In this section we want to use our characterisation of Rieffel proper actions
to show that certain examples of Exel-Meyer proper actions  as discussed by
Meyer in \cite{Meyer:generalised_Fixed}  are actually Rieffel proper. This will provide
us with examples in which a given $G$\nb-algebra $(A,\alpha)$ has infinitely many
different dense subspaces $\mathcal R_i\subseteq A$, $i\in I$,  such that
$\alpha:G\to \Aut(A)$ is Rieffel proper with respect to all $\Rel_i$
as in Proposition \ref{prop:CharacterizationRieffelProperActions}
but with pairwise non-isomorphic generalised fixed-point algebras. As we shall see, such examples
can even occur if all such Rieffel proper actions are saturated. In this case all fixed-point algebras
have to be  Morita-equivalent, since they are Morita equivalent to $A\rtimes_rG$.

We start with a brief introduction to the theory of Exel-Meyer proper actions
as defined by Exel and Meyer in \cite{Exel:SpectralTheory} and \cite{Meyer:Equivariant, Meyer:generalised_Fixed}.
Such actions provide generalisations of Rieffel proper actions which still allow the
construction of generalised fixed-point algebras $A^G$.

 Let $(B,\beta)$ be a $G$\nb-algebra. In what follows we  realise the left regular representation $\lambda^B\colon C_c(G,B)\to \Lb(L^2(G,B))$ by the formula
\begin{equation}\label{eq-regular}\lambda^B_\varphi(f)|_t\defeq \int_G\Delta(s)^{-1}\beta_{t^{-1}}(\varphi(ts^{-1}))f(s)ds\quad \varphi\in C_c(G,B), f\in L^2(G,B).
\end{equation}
Recall that the reduced crossed product $B\rtimes_{\beta,r}G$ can be defined
as the closure of $\lambda^B(C_c(G,B))$ inside $\Lb(L^2(G,B))$.
More generally, we say that a continuous  function $\varphi:G\to B$ is a \emph{bounded symbol} if the integral
operator $\lambda^B_\varphi$  of \eqref{eq-regular}, which makes always sense for $f\in C_c(G,B)$,  extends to an
adjointable operator in $\Lb(L^2(G,B))$.
Such an integral operator is  called \emph{Laurent operator} with symbol $\varphi$, a terminology introduced in \cite{Exel:SpectralTheory}.
We shall often identify bounded symbol functions $\varphi$ with the corresponding operator $\lambda^B_\varphi$.

Assume now that $(\E,\gamma)$  is a $G$\nb-equivariant Hilbert $B$-module.
Given $\xi\in \E$, we define linear operators:
$$\kket{\xi}\colon \contc(G,B)\to \E,\quad \kket{\xi}f\defeq \int_G\Delta(t)^{-1/2}\gamma_{t^{-1}}(\xi)f(t)dt,$$
and
$$\bbra{\xi}\colon \E\to \cont(G,B),\quad \bbra{\xi}\eta(t)\defeq \Delta(t)^{-1/2}\braket{\gamma_{t^{-1}}(\xi)}{\eta}.$$
We say that $\xi$ is \emph{square integrable} if $\kket{\xi}$ extends to an adjointable operator $L^2(G,B)\to \E$.
This is equivalent to saying that the continuous function $\bbra{\xi}\eta$ lies in $L^2(G,B)$;
and in this case $\bbra{\xi}$ is automatically the adjoint operator of $\kket{\xi}$.
It is easy to see that $\kket{\xi}$ and $\bbra{\xi}$ are $G$\nb-equivariant operators with respect to the given $G$\nb-action $\gamma$ on $\E$ and the
$\beta$-compatible $G$\nb-action $(\rho\otimes\beta)_t(f)(s)\defeq \Delta(t)^{1/2}\beta_t(f(st))$ on the Hilbert $B$-module $L^2(G,B)$.

The $G$\nb-equivariant Hilbert $B$-module $(\E,\gamma)$ is called \emph{square integrable} if the space $\E_\si$ of square-integrable
elements is dense in $\E$. The theory of square-integrable modules is developed in detail by Meyer in \cite{Meyer:Equivariant, Meyer:generalised_Fixed}. Actually, in the papers \cite{Meyer:Equivariant,Meyer:generalised_Fixed} the modular function and the inverses do not show up in the definition of $\kket{\xi}$ or $\bbra{\xi}$, i.e., $\gamma_{t}(\xi)$ appears in place of $\Delta(t)^{-1/2}\gamma_{t^{-1}}(\xi)$ above.
The reason is that Meyer uses the left regular representation $\lambda$ as the \emph{standard} representation of $G$ on $L^2(G)$, while we use  the  right regular representation $\rho$ instead. The above formulas translate into Meyer's formulas via the
unitary intertwiner $U:L^2(G)\to L^2(G); (Uf)(t)\defeq \Delta(t)^{-1/2}f(t^{-1})$ between $\lambda$ and $\rho$.
The operators $\kket{\xi}\circ U$ and $U\circ\bbra{\xi}$ are then exactly the operators used by Meyer in \cite{Meyer:generalised_Fixed}. Our convention follows the paper \cite{Exel:SpectralTheory} by Exel, from which the basic ideas in \cite{Meyer:generalised_Fixed} are built.

A $G$\nb-algebra $(A,\alpha)$, when viewed as a $G$\nb-equivariant Hilbert $A$-module in the canonical way, is square-integrable if and only if it is integrable in the sense of Definition \ref{def-integrable}.
More generally, it is proved in \cite{Meyer:Equivariant} that a $G$\nb-equivariant Hilbert $B$-module $(\E,\gamma)$ is square integrable if and only if the $G$\nb-algebra of compact operators $A=\K(\E)$ with $G$\nb-action $\alpha=\Ad\gamma$ is integrable.
Moreover, if $\xi,\eta\in \E_\si$, then
the rank-one operator $\ket{\xi}\bra{\eta}\in \K(\E)$ defined by $\ket{\xi}\bra{\eta}\zeta=\xi\braket{\eta}{\zeta}_{B}$ is $\alpha$-integrable and
$$\int_G^\su\alpha_t(\ket{\xi}\bra{\eta})\, dt=\kket{\xi}\bbra{\eta},$$
where $\kket{\xi}\bbra{\eta}$ denotes the composition  $\kket{\xi}\circ\bbra{\eta}\in \Lb_A^G(A)=\M(A)^G$.

To see the connection of integrability with Rieffel properness it is important to
describe in detail the composition
 $\bbra{\xi}\circ\kket{\eta}$ for all $\xi,\eta\in \E_{\si}$, which is a $G$\nb-equivariant operator on $L^2(G,B)$, that is, an element of the
fixed-point algebra $\M(B\otimes\K(L^2G))^G=\Lb^G_B(L^2(G,B))$.
In general, if  $\xi,\eta\in \E$, we define $\bbraket{\xi}{\eta}$ to be the continuous function $t\mapsto \Delta(t)^{-1/2}\braket{\xi}{\gamma_t(\eta)}$.
We then compute for $f\in C_c(G,B)$:
\begin{alignat*}{2}
\bbra{\xi}\circ\kket{\eta}(f)|_t&=\Delta(t)^{-1/2}\braket{\gamma_{t^{-1}}(\xi)}{\kket{\eta}f}\\
&=\int_G\Delta(ts)^{-1/2}\braket{\gamma_{t^{-1}}(\xi)}{\gamma_{s^{-1}}(\eta)}f(s)\, ds\\
&=\int_G\Delta(s)^{-1}\beta_{t^{-1}}(\Delta(ts^{-1})^{-1/2}\braket{\xi}{\gamma_{ts^{-1}}(\eta)})f(s)\, ds\\
&=\int_G\Delta(s)^{-1}\beta_{t^{-1}}(\bbraket{\xi}{\eta}(ts^{-1}))f(s)\, ds\\
&=\lambda^B_{\bbraket{\xi\,}{\,\eta}}(f)|_t,
\end{alignat*}
so that we get the equation $\bbraket{\xi}{\eta}=\bbra{\xi}\circ\kket{\eta}$ for $\xi,\eta\in \E_\si$ if we identify
the symbol $\bbraket{\xi}{\eta}=[t\mapsto \Delta(t)^{-1/2}\braket{\xi}{\gamma_t(\eta)}]$ with the corresponding Laurent operator
$\lambda^B_{\bbraket{\xi}{\eta}}$.
It is easy to see that for every bounded symbol function $\varphi\in \cont(G,B)$, $\lambda^B_\varphi$ is always an element of the fixed-point \cstar{}subalgebra $\Lb^G(L^2(G,B))=\M(B\otimes\K(L^2G))^G$. In particular, $B\rtimes_{\beta,r}G\sbe \M(B\otimes\K(L^2G))^G$, where $G$ acts on  $B\otimes\K(L^2G)$ via $\beta\otimes\Ad\rho$.
But for a bounded symbol $\varphi$, the Laurent operator $\lambda^B_\varphi$ need not be in $B\rtimes_{\beta,r}G$.
The following definition is the main point of the above discussion:

\begin{definition}[Exel-Meyer]
We say that a subspace $\Rel\sbe \E_\si$ of a $G$\nb-equivariant Hilbert $B$-module $(\E,\gamma)$ is \emph{relatively continuous} if $\bbraket{\Rel}{\Rel}\sbe B\rtimes_{\beta,r}G$.
We say that the action $(\E,\gamma)$ is \emph{Exel-Meyer proper} if there is a dense relatively continuous subspace $\Rel\sbe \E$.

In particular, if $(A,\alpha)$ is a $G$\nb-C*-algebra, we say that $\alpha$ is  \emph{Exel-Meyer proper} if there is a
dense relatively continuous subspace $\Rel\sbe A_{\si}$, where we view $(A,\alpha)$ as a $G$\nb-equivariant Hilbert $A$-module.
\end{definition}

In the above form, these actions were introduced by Meyer in \cite{Meyer:generalised_Fixed}, but the essential ideas are already contained in Exel's paper \cite{Exel:SpectralTheory} who defined relative continuity for actions of locally compact \emph{abelian} groups on \cstar{}algebras only.
The Exel-Meyer-proper actions are, in a sense, the most general proper actions which allow the construction of generalised fixed-point algebras:

\begin{definition}
The \emph{generalised fixed-point algebra} associated to a relatively continuous subspace $\Rel\sbe \E$ is, by definition, the \cstar{}subalgebra $\Fix(\E,\Rel)$ of $\Lb^G(\E)=\M(\K(\E))^G$ generated by the operators $\kket{\xi}\bbra{\eta}=\int_G^\su\ket{\gamma_t(\xi)}\bra{\gamma_t(\eta)}\, dt$ with $\xi,\eta\in \Rel$.
\end{definition}

The generalised fixed-point algebra $\Fix(\E,\Rel)$ is always Morita equivalent to an ideal in $B\rtimes_{\beta,r}G$.
The construction of the bimodule $\F(\Rel)$ implementing this equivalence can be performed essentially in the same way as we did for Rieffel proper actions: take $\tilde\Rel\defeq \Rel*\contc(G,B)$ and endow it with the usual right $\contc(G,B)$-convolution action and the inner product $\bbraket{\cdot}{\cdot}$ and complete it to a right Hilbert $B\rtimes_{\beta,r}G$-module $\F(\Rel)$. The algebra of compact operators on $\F(\Rel)$ is then canonically isomorphic to $\Fix(\E,\Rel)$ with left action induced by the left action of $\Lb(\E)$ on $\Rel$ and the left-inner product $(\xi,\eta)\mapsto \kket{\xi}\bbra{\eta}$.
The details can be found in \cite{Meyer:generalised_Fixed}*{\S6}.

If $(\E, \gamma)$ is a $G$\nb-equivariant Hilbert $(B,\beta)$-module, there is an important connection  between
relatively continuous dense subsets $\mathcal R_\E\subseteq \E_{\si}$ and relatively continuous dense subsets
${\mathcal R}_\K\subseteq \K(\E)_{\si}$ given as follows:
If $\mathcal R_\E\subseteq \E_{\si}$ is relatively continuous, then it is shown in \cite{Meyer:generalised_Fixed}*{Corollary 7.2} that
$$\mathcal R_\K:=\spn\{\ket{\xi}\bra{\eta}: \xi\in \mathcal R_\E, \eta\in \E\}$$
is a dense relatively continuous subset of $\K(\E)_{\si}$ with respect to the action $\alpha=\Ad\gamma$ such that the
corresponding Hilbert $\K(\E)\rtimes_{\Ad\gamma, r}G$-module $\F(\mathcal R_\K)$ satisfies
$$ \F(\mathcal R_\K)\cong \F(\mathcal R_\E)\otimes_{B\rtimes_rG}(\E^*\rtimes_r G).$$
In particular, since $\K(\E^*\rtimes_rG)$ is an ideal in $B\rtimes_rG$, it follows that
\begin{equation}\label{eq-Fixed}
\begin{split}
\K(\E)^G:&=\Fix(\K(\E), \mathcal R_\K)=\K(\F(\mathcal R_\K))=\K\big(\F(\mathcal R_\E)\otimes_{B\rtimes_rG}(\E^*\rtimes_r G)\big)\\
&=\K(\F(\mathcal R_\E))=\Fix(\E, \mathcal R_\E).
\end{split}
\end{equation}
Conversely, if $(A,\alpha)$ is any square{-}integrable $G$\nb-C*-algebra, $\phi:A\to\L(\E)$ is a nondegenerate $G$\nb-equivariant
*-homomorphism, and $\mathcal R_A\subseteq A_{\si}$ is a dense relatively continuous subspace of $A$, then it is shown in
\cite{Meyer:generalised_Fixed}*{Corollary 7.1} that
$\mathcal R_\E:=\phi(\mathcal R_A)\E$ is a dense relatively continuous subspace of $\E_\si$ such that
$$\F(\mathcal R_\E)\cong\F(\mathcal R_A)\otimes_{A\rtimes_rG}(\E\rtimes_r G)$$
as Hilbert $B\rtimes_rG$-modules. In particular, if $A=\K(\E)$ we get
$$
\Fix(\E, \mathcal R_\E)=\K(\F(\mathcal R_\E))\cong \K(\F(\mathcal R_\K))=\Fix(\K(\E), \mathcal R_\K)
=\K(\E)^G.
$$

We now want to discuss a similar correspondence between subsets $\mathcal R_\E\subseteq \E$ and
 $\mathcal R_\K\subset \K(\E)$ which induce Rieffel properness as in Proposition \ref{prop:CharacterizationRieffelProperActions}.
 Motivated by the above discussion on relatively continuous subsets, we introduce the following:

 \begin{definition}\label{def-relativeL1}
 Let $(\E,\gamma)$ be a $G$\nb-equivariant Hilbert $(B,\beta)$-module. Then a subset $\mathcal R\subseteq \E$ is said
 to be {\em relatively $L^1$}  if for each $\xi,\eta\in \mathcal R$ the function
 $\bbraket{\xi}{\eta}=\big[s\mapsto \Delta(s)^{-1/2}\braket{\xi}{\gamma_s(\eta)}_B\big]$ lies
 in $L_\Delta^1(G,B)$.
 \end{definition}

 Of course, if  $(A,\alpha)$ is a $G$\nb-algebra, viewed as a $G$\nb-equivariant Hilbert $(A,\alpha)$-module,
 then a subset $\mathcal R\subseteq A$ is relatively $L^1$ in the sense of the above definition if and only if
 it satisfies condition (P1) of Proposition \ref{prop:CharacterizationRieffelProperActions}. Hence
 there exists a dense relatively $L^1$ subspace $\mathcal R\subseteq A$ if and only if the action
 $\alpha:G\to\Aut(A)$ is Rieffel proper.
  Motivated by Proposition \ref{prop:CharacterizationRieffelProperActions} we introduce the following:
 \begin{definition}\label{def-propermodule}
 Suppose that $(\E,\gamma)$ is a $G$\nb-equivariant Hilbert $(B,\beta)$-module. We  say that
 the action $\gamma: G\to \Aut(\E)$ is Rieffel proper if there exists a dense
 relatively $L^1$-subspace $\mathcal R\subseteq \E$.
 \end{definition}

 It follows directly from the definition and the fact that $L^1(G,B)\subseteq B\rtimes_rG$ that every dense relatively $L^1$\nb-subspace $\Rel\subseteq\E$
 is also relatively continuous in the sense of Exel-Meyer. Therefore it follows from the
 results of Meyer in \cite{Meyer:generalised_Fixed} that $\tilde{\mathcal R}=\mathcal R*C_c(G,B)$ completes to
 a Hilbert $B\rtimes_{\beta,r}G$-module $\F(\mathcal R)$ with respect to the $B\rtimes_rG$-valued inner
 product
 $$\bbraket{\xi}{\eta}=\Big( s\mapsto \Delta(s)^{-1/2}\braket{\xi}{\gamma_s(\eta)}\Big)\in L^1(G,B)\subseteq B\rtimes_rG.$$
Since $C_c(G,B)$ is dense in $L_\Delta^1(G,B)\subseteq B\rtimes_rG$,
it follows that every element in $\mathcal R* L^1_\Delta(G,B)$ can be approximated by
 elements in $\mathcal R* C_c(G,B)$ in the norm induced by the $B\rtimes_rG$-valued inner product $\bbraket{\cdot}{\cdot}$.
 Hence, if $\mathcal R_A\subseteq A$ is a dense relatively $L^1$\nb-subspace of a $G$\nb-C*-algebra $(A,\alpha)$,
 then $\F(\mathcal R_A)$ coincides with the module as discussed preceding Proposition \ref{prop-uniquemodule}.
 In particular, the corresponding fixed-point algebras coincide.

\begin{lemma}\label{lem-E-to-KE}
Suppose that $(\E,\gamma)$ is a $G$\nb-equivariant Hilbert $(B,\beta)$-module and that
$\mathcal R_\E\subseteq \E$ is a dense relatively $L^1$\nb-subspace of $\E$.
Then
$$\mathcal R_\K:=\spn\{\ket{\xi}\bra{\eta}: \xi\in \mathcal R_\E, \eta\in \E\}$$
is a dense relatively  $L^1$\nb-subspace of $\K(\E)$ such that
$$ \F(\mathcal R_\K)\cong \F(\mathcal R_\E)\otimes_{B\rtimes_rG}(\E^*\rtimes_r G).$$
Conversely, if $(A,\alpha)$ is a $G$\nb-C*-algebra such that there exists a nondegenerate $G$\nb-equivariant
*-homomorphism $\phi:A\to\L(\E)$ and if $\mathcal R_A$ is a dense relatively $L^1$\nb-subspace of $A$,
then $\mathcal R_\E=\phi(\mathcal R_A)\E$ is a dense relatively $L^1$-subspace of $\E$ such
that
$$\F(\mathcal R_\E)\cong\F(\mathcal R_A)\otimes_{A\rtimes_rG}(\E\rtimes_r G).$$
\end{lemma}
\begin{proof} For the proof we only need to check the $L^1$ conditions for $\mathcal R_\K$ and $\mathcal R_\E$,
respectively, since everything else will follow from the corresponding results for Exel-Meyer proper actions
as discussed above.

So suppose that $\mathcal R_\E\subseteq \E$ is a relatively $L^1$\nb-subspace of $\E$.
Recall that $\ket{\xi}\bra{\eta}\in \K(\E)$ is the operator defined by
$$\ket{\xi}\bra{\eta}\zeta=\xi\cdot \braket{\eta}{\zeta}_B$$
for all $\zeta\in \E$. Then a short computation shows that for $\xi,\xi', \eta, \eta'\in \E$ we get the equation
$$\ket{\xi}\bra{\eta}\circ  \ket{\xi'}\bra{\eta'}=\ket{\xi\cdot \braket{\eta}{\xi'}_B}\bra{\eta'}.$$
Using this we get for all $\xi,\xi'\in \mathcal R_\E$ and $\eta,\eta'\in \E$ that
\begin{align*}
(\ket{\xi}\bra{\eta})^*\circ \Ad\gamma_s(\ket{\xi'}\bra{\eta'})&=
\ket{\eta}\bra{\xi}\circ  \ket{\gamma_s(\xi')}\bra{\gamma_s(\eta')}\\
&=\ket{\eta\cdot \braket{\xi}{\gamma_s(\xi')}_B}\bra{\gamma_s(\eta')},
\end{align*}
from which it follows that
$$
\left\|(\ket{\xi}\bra{\eta})^*\circ \Ad\gamma_s(\ket{\xi'}\bra{\eta'})\right\|\leq \|\braket{\xi}{\gamma_s(\xi')}_B\|\|\eta\|\|\eta'\|.
$$
Since $\xi,\xi'\in \mathcal R_\E$, the right hand side of this equation (and hence also the left hand side) defines a function in $L^1_\Delta(G)$.
Hence $\mathcal R_\K$ is relatively $L^1$.

Conversely, if $\mathcal R_A\subseteq A$ is relatively $L^1$ and if $a,b\in \mathcal R_A$ and $\xi,\eta\in \E$, then
$$
s\mapsto \|\braket{\phi(a)\xi}{\gamma_s(\phi(b)\eta)}_B\|=\|\braket{\phi(\alpha_s(b^*)a)\xi}{\gamma_s(\eta)}_B\|\leq
\|\alpha_s(b^*)a\| \|\xi\|\|\eta\|
$$
lies in $L^1_\Delta$, hence $\mathcal R_\E=\phi(\mathcal R_A)\E$ is relatively $L^1$.
\end{proof}

We shall now use the above correspondence between  dense relatively $L^1$\nb-subspaces of $\E$ and
dense relatively $L^1$\nb-subspaces of $\K(\E)$ to study certain examples of Meyer in
\cite{Meyer:generalised_Fixed}*{\S 8} in the context of Rieffel proper actions.
For this we let $(B,\beta)=(\C,\id)$, the complex numbers with trivial $G$\nb-action.
Then a $G$\nb-equivariant Hilbert $\C$-module is a pair $(\H, u)$ where $\H$ is a Hilbert space
and $u:G\to \U(\H)$ is a unitary representation of $G$ on $\H$.
Using the above lemma it follows that, if
we can find  dense relatively $L^1$-subspaces $\mathcal R^i\subseteq \H$, $i=1,2$,
with non-isomorphic fixed-point algebras $\Fix(\H,\mathcal R^i)=\K(\F(\mathcal R^i))$,
then the action $\Ad u:G\to \Aut(\K(\H))$ will be Rieffel proper with respect to
the  corresponding subspaces $\mathcal R^i_{\K}=\ket{\mathcal R_i}\bra{\H}\subseteq \K(\H)$, $i=1,2$,
 with non-isomorphic
modules $\F(\mathcal R^i_\K)$ and non-isomorphic fixed-point algebras
$$\K(\H)^G_i=\K(\F(\mathcal R^i_\K))\cong \Fix(\H, \mathcal R^i),$$
which follows from Lemma \ref{lem-E-to-KE} together with  (\ref{eq-Fixed}).

Following Meyer in \cite{Meyer:generalised_Fixed}*{\S 8} we look at the particular case
where $G=\Z^k$ and $(\H,u)=(\ell^2(\Z^k)^n, \rho^n)$, where $\rho^n$ denotes the $n$-fold
direct sum of the right regular representation $\rho:\Z^k\to \U(\ell^2(\Z^k))$ (here we replace $\lambda$,
used in \cite{Meyer:generalised_Fixed}*{\S 8}, by $\rho$ to make the example compatible with our
general policies as explained at the beginning of this section).

Assume that $\mathcal R\subseteq \ell^2(\Z^k)^n$ is a dense relatively
$L^1$\nb-subspace. Then the corresponding $C^*(\Z^k)$-valued inner
product $\bbraket{\xi}{\eta}$ on $\mathcal R$ is given by the function
$$m\mapsto \bbraket{\xi}{\eta}(m)= \braket{\xi}{\rho^n_m(\eta)}= \sum_{i=1}^n\left( \sum_{\nu\in \Z^k}\overline{\xi_i}(\nu)\eta_i(\nu+m)\right)
=\sum_{i=1}^n \overline{\xi}_i*\eta_i(m),$$
for $\xi,\eta\in \mathcal R$, where $\xi_i,\eta_i$ denote the $i$-th components of $\xi$ and $\eta$, respectively.
Using Fourier transform everywhere turns $\ell^2(\Z^k)^n$ into $L^2(\T^k)^n$ and $C^*(\Z^k)$ into $C(\T^k)$,
and the $C^*(\Z^k)$-valued inner product on $\mathcal R\subseteq \ell^2(\Z^k)^n$ is transformed into
the $C(\T^k)$-valued inner product on the subspace $\widehat{\mathcal R}=\{\widehat{\xi}: \xi\in \mathcal R\}\subset L^2(\T^k)^n$
with inner product
$$ \bbraket{\widehat{\xi}}{\widehat{\eta}}_{C(\T^k)}(z)=\sum_{i=1}^n \overline{\widehat{\xi}\,}_i\cdot\widehat{\eta}_i({z}).$$
The
Plancherel isomorphism $\ell^2(\Z^k)^n\cong L^2(\T^k)^n$
 intertwines the right regular representation $\rho$ of $\Z^k$ on $\ell^2(\Z^k)$
with  the unitary representation $\widehat\rho$ given by
$$(\widehat{\rho}_m \xi)(z)=z^{-m}\xi(z)\quad \forall z\in \T^k, m\in \Z^k,$$
where $z^{-m}:=z_1^{-m_1}\cdot z_2^{-m_2}\cdots z_k^{-m_k}\in \T$.
The dense relatively  continuous subspaces $\widehat{\mathcal R}\subseteq L^2(\T^k)^n$
with respect to the unitary representation $\widehat{\rho}^n$ have been studied in
detail by Meyer in \cite{Meyer:generalised_Fixed}*{\S 8} and by Buss and Meyer
in \cite{Buss-Meyer:Continuous}. In particular, we are interested in the following
two special examples:

\begin{example}\label{ex-subset}
Consider the case $n=1$. Let $S\subseteq \T^k$ be an open dense subspace of $\T^k$ with
full Haar measure. Then we get the following chain of dense subsets in $L^2(\T^k)\cong \ell^2(\Z^k)$:
$$\widehat{\mathcal R_S}=C_c^\infty(S)\subseteq C_c(S)\subseteq L^2(S)=L^2(\T^k),$$
where $C_c^\infty(S)$ denotes the set smooth functions with compact supports on $S$.
Then the $C(\T^k)$-valued inner products $\bbraket{\xi}{\eta}_{C(\T^k)}=\overline{\xi}\cdot\eta$ all lie
in
$$C_c^\infty(S)\subseteq I_S:=C_0(S)\subseteq C(\T^k).$$
Since the inverse Fourier-transform of any smooth function on $\T^k$ lies in $\ell^1(\Z^k)$
we see that the preimage $\mathcal R_S$ of $\widehat{\mathcal R_S}$ under the Fourier transform
is a dense relatively $L^1$\nb-subspace of $\ell^2(\Z^k)$.
As in \cite{Meyer:generalised_Fixed}*{\S 8} one checks that the module   $\F(\widehat{\mathcal R_S})$
is isomorphic to the standard $C_0(S)-C_0(S)$ equivalence bimodule $C_0(S)$, thus the
generalized fixed-point algebra $\K(\ell^2(\Z^k))^{\Z^k, \Ad\rho}=\Fix(\ell^2(\Z^k),\mathcal R_S)$ is
also isomorphic to $C_0(S)$. Since there exist infinitely many non-homeomorphic open dense subsets
$S\subseteq \T^k$,  there exist infinitely
many non-isomorphic generalised fixed-point algebras for the
Rieffel proper action  $\Ad\rho:\Z^k\to\Aut(\K(\ell^2(\Z^k))$.
\end{example}

One might observe that in the above example only the case $S=\T^k$ provides a
structure of a {\em saturated} Rieffel proper action in which the corresponding
Hilbert $C^*(\Z^k)$-module $\F(\mathcal R_S)$ is full. But the following slight alteration
of another example given by Meyer in \cite{Meyer:generalised_Fixed}*{\S 8} yield
examples of different  saturated Rieffel proper structures with non-isomorphic fixed-point algebras:

\begin{example}\label{ex-bundle}
We are now looking at structures $\mathcal R\subseteq \ell^2(\Z^k)^n$. Again we dualise
in order to consider subspaces $\widehat{\mathcal R}\subseteq L^2(\Z^k)^n$.
Suppose  that $p:\mathcal V\to \T^k$ is an $n$-dimensional complex hermitian vector bundle
over $\T^k$. Then the Hilbert space $L^2(\T^k, \mathcal V)$
of $L^2$-sections of $\mathcal V$ is isomorphic to $L^2(\T^k)^n$.
The easiest way to see this is to choose a partition $\{A_i:1\leq i\leq m\}$ of $\T^k$
with measurable sets $A_i\subseteq \T^k$ such that $\mathcal V|_{A_i}\cong A_i\times \C^n$.
Then both spaces are isomorphic to $\oplus_{i=1}^m L^2(A_i)^n$.
As pointed out in \cite{Meyer:generalised_Fixed}*{p.~190}, the subspace
$\widehat{\mathcal R}_{\mathcal V}:=C(\T^k,\mathcal V)\subseteq L^2(\T^k, \mathcal V)$
of continuous sections of the vector bundle $\mathcal V$ is a dense relatively continuous subset
of $L^2(\T^k, \mathcal V)\cong L^2(\T^k)^n$ with full $C(\T^k)\cong C^*(\Z^k)$-valued inner product,
which makes it into a $C(\T^k,\End(\mathcal V))-C(\T^k)$ equivalence bimodule.
Hence the inverse image $\mathcal R_{\mathcal V}$ of $\widehat{\mathcal R}_{\mathcal V}$ under
the Plancherel isomorphism is a dense relatively continuous subset of $\ell^2(\Z^k)^n$ with
corresponding generalised fixed-point algebra
$$\Fix(\ell^2(\Z^k)^n,\mathcal R_{\mathcal V})\cong C(\T^k, \End(\mathcal V)).$$
Now, if $p:\mathcal V\to\T^k$ is a smooth vector bundle, we can look at the subspace
$$\mathcal R_{\mathcal V}^\infty:=\{\xi\in \ell^2(\Z^k)^n:\widehat\xi\in C^\infty(\T^k,\mathcal V)\}.$$
Then the $C(\T^k)$-valued inner product $\bbraket{\widehat{\xi}}{\widehat{\eta}}_{C(\T^k)}$
for $\xi,\eta\in \mathcal R_{\mathcal V}^\infty$ lies in $C^\infty(\T^k)$ and therefore
$\bbraket{\xi}{\eta}=\F^{-1}\big(\bbraket{\widehat{\xi}}{\widehat{\eta}}_{C(\T^k)}\big)\in  \ell^1(\Z^k)$,
where $\F^{-1}$ denotes inverse Fourier transform.
 Hence
$\mathcal R_{\mathcal V}^\infty$ is a dense relatively $L^1$\nb-subspace of $\ell^2(\Z^k)^n$.
Since every continuous section $\widehat{\xi}:\T^k\to\mathcal V$ can be approximated
by  smooth sections with respect to the $C(\T^k)$-valued inner product,
the module $\mathcal F(\mathcal R_{\mathcal V}^\infty)$ coincides with $\F(\mathcal R_{\mathcal V})$,
hence the corresponding generalised fixed-point algebra is also $C(\T^k,\End(\mathcal V))$.
In particular, in case of the trivial bundle $\mathcal V=\T^k\times \C^n$ the fixed-point algebra
will be $C(\T^k, {\Mat_n}(\C))$.

Thus, in order to find different  structures for saturated Rieffel properness of the action
$\Ad\rho^n:\Z^k\to \Aut(\K(\ell^2(\Z^k)^n))$
with non-isomorphic fixed-point algebras, it suffices to find smooth $n$-dimensional vector bundles
$\mathcal V$ over $\T^k$ such that $C(\T^k, \End(\mathcal V))$ is not isomorphic to $C(\T^k, {\Mat_n}(\C))$.
As pointed out on the bottom of  \cite{Meyer:generalised_Fixed}*{p.~190}, if $k=n=2$ and if
$L_1$ is a smooth realisation of the line bundle on $\T^2$ corresponding to the generator
of $H^2(\T^2,\Z)$, which is given by gluing $(\T\times [0,1])\times \C$ at the endpoints
with respect to the smooth function
$$\T\times\{0\}\times\C\to \T\times \{1\}\times \C: (z, 0, v)\mapsto (z, 1, zv),$$
then for $\mathcal V= \C\oplus L_1$ we have $C(\T^n,\End(\mathcal V))\not{\!\!\cong} \; C(\T^n,{\Mat_2}(\C))$.
More generally, if we let $L_m$ denote the line bundle  over $\T^2$ given by gluing $(\T\times [0,1])\times \C$ at the
endpoints with respect to the smooth function $(z, 0, v)\mapsto (z, 1, z^mv)$, for $m\in \Z$, and putting
$\mathcal V_m= \C\oplus L_m$, then similar arguments as used by Meyer show that
the corresponding  generalised fixed-point algebras $C(\T^2,\End(\mathcal V_n))$ are
pairwise non-isomorphic. This yields infinitely many relatively $L^1$\nb-subspaces
 $\mathcal R_m\subseteq \ell^2(\Z^2)^2$ for the $\Z^2$-algebra $(\K(\ell^2(\Z^2)^2) , \Ad\rho^2)$
such that the corresponding fixed-point algebras are pairwise
non-isomorphic.
\end{example}

\vskip 0,5pc


\begin{bibdiv}
  \begin{biblist}

\bib{Huef-Kaliszewski-Raeburn-Williams:Naturality_Rieffel}{article}{
  author={an Huef, Astrid},
  author={Kaliszewski, Steven P.},
  author={Raeburn, Iain},
  author={Williams, Dana P.},
  title={Naturality of Rieffel's Morita equivalence for proper actions},
  journal={Algebr. Represent. Theory},
  volume={14},
  date={2011},
  number={3},
  pages={515--543},
  issn={1386-923X},
  review={\MRref {2785921}{}},
  doi={10.1007/s10468-009-9201-2},
}

\bib{anHuef-Raeburn-Williams:Symmetric}{article}{
  author={an Huef, Astrid},
  author={Raeburn, Iain},
  author={Williams, Dana P.},
  title={A symmetric imprimitivity theorem for commuting proper actions},
  journal={Canad. J. Math.},
  volume={57},
  date={2005},
  number={5},
  pages={983--1011},
  issn={0008-414X},
  review={\MRref {2164592}{2006f:46067}},
}

\bib{Bag-Klep}{article}{
    AUTHOR = {Baggett, Larry},
    author={Kleppner, Adam},
     TITLE = {Multiplier representations of abelian groups},
   JOURNAL = {J. Functional Analysis},
    VOLUME = {14},
      YEAR = {1973},
     PAGES = {299--324},
   MRCLASS = {22D12},
  MRNUMBER = {0364537 (51 \#791)},
MRREVIEWER = {R. C. Busby},
}

\bib{Buss-Echterhoff:Exotic_GFPA}{article}{
  author={Buss, Alcides},
  author={Echterhoff, Siegfried},
  title={Universal and exotic generalised fixed-point algebras for weakly proper actions and duality},
  status={eprint},
  note={\arxiv{1304.5697}, to appear in Indiana Univ. Math. J.},
  date={2013},
}

  {\bib{Buss-Echterhoff:Imprimitivity}{article}{
  author={Buss, Alcides},
  author={Echterhoff, Siegfried},
 title={Imprimitivity theorems for weakly proper actions and duality},
  status={eprint},
  note={\arxiv {1305.5100}, to appear in Ergodic theory \& dynamical Systems},
  year={2013},
 }
}
{ \bib{Buss-Echterhoff:Mansfield}{article}{
  author={Buss, Alcides},
  author={Echterhoff, Siegfried},
 title={Weakly proper group actions, Mansfield's imprimitivity and twisted Landstad duality},
  status={eprint},
  note={\arxiv {1310.3934}, to appear in Trans. Amer. Math. Soc.},
  year={2013},
 }
 }

\bib{Buss-Meyer:Continuous}{article}{
  author={Buss, Alcides},
  author={Meyer, Ralf},
  title={Continuous spectral decompositions of Abelian group actions on \(C^*\)\nobreakdash-algebras},
  journal={J. Funct. Anal.},
  volume={253},
  date={2007},
  number={2},
  pages={482--514},
  issn={0022-1236},
  review={\MRref{2370086}{2009f:46091}},
  doi={10.1016/j.jfa.2007.04.009},
}

\bib{DE}{book}{
    AUTHOR = {Deitmar, Anton}
    author={Echterhoff, Siegfried},
     TITLE = {Principles of harmonic analysis},
    SERIES = {Universitext},
 PUBLISHER = {Springer, New York},
      YEAR = {2009},
     PAGES = {xvi+333},
      ISBN = {978-0-387-85468-7},
   MRCLASS = {43-01 (22Dxx 22E30 46J10 46L05)},
  MRNUMBER = {2457798 (2010g:43001)},
MRREVIEWER = {Krishnan Parthasarathy},
}

\bib{Exel:Unconditional}{article}{
  author={Exel, Ruy},
  title={Unconditional integrability for dual actions},
  journal={Bol. Soc. Brasil. Mat. (N.S.)},
  volume={30},
  number={1},
  date={1999},
  pages={99--124},
  issn={0100-3569},
  review={\MRref{1686980}{2000f:46071}},
  doi={10.1007/BF01235677},
}

\bib{Exel:SpectralTheory}{article}{
  author={Exel, Ruy},
  title={Morita--Rieffel equivalence and spectral theory for integrable automorphism groups of $C^*$\nobreakdash-algebras},
  journal={J. Funct. Anal.},
  volume={172},
  date={2000},
  number={2},
  pages={404--465},
  issn={0022-1236},
  doi={10.1006/jfan.1999.3537},
  review={\MRref{1753180}{2001h:46104}},
}

\bib{Huse}{book}{
    AUTHOR = {Husemoller, Dale},
     TITLE = {Fibre bundles},
    SERIES = {Graduate Texts in Mathematics},
    VOLUME = {20},
   EDITION = {Third edition},
 PUBLISHER = {Springer-Verlag, New York},
      YEAR = {1994},
     PAGES = {xx+353},
      ISBN = {0-387-94087-1},
   MRCLASS = {55-01 (19Lxx 55-02 55Rxx 57R20 57R22)},
  MRNUMBER = {1249482 (94k:55001)},
MRREVIEWER = {Jo{\v{z}}e Vrabec},
       DOI = {10.1007/978-1-4757-2261-1},
       URL = {http://dx.doi.org/10.1007/978-1-4757-2261-1},
}

\bib{Kleppner:Multipliers}{article}{
    AUTHOR = {Kleppner, Adam},
     TITLE = {Multipliers on abelian groups},
   JOURNAL = {Math. Ann.},
  FJOURNAL = {Mathematische Annalen},
    VOLUME = {158},
      YEAR = {1965},
     PAGES = {11--34},
      ISSN = {0025-5831},
   MRCLASS = {22.65},
  MRNUMBER = {0174656 (30 \#4856)},
MRREVIEWER = {J. M. G. Fell},
}

\bib{Kustermans-Vaes:Weight}{article}{
  author={Kustermans, Johan},
  author={Vaes, Stefaan},
  title={Weight theory for \(C^*\)\nobreakdash-algebraic quantum groups},
  status={eprint},
  note={\arxiv{math/9901063}},
  date={1999},
}

\bib{Mackey:UnitaryI}{article}{
    AUTHOR = {Mackey, George W.},
     TITLE = {Unitary representations of group extensions. {I}},
   JOURNAL = {Acta Math.},
  FJOURNAL = {Acta Mathematica},
    VOLUME = {99},
      YEAR = {1958},
     PAGES = {265--311},
      ISSN = {0001-5962},
   MRCLASS = {46.00 (22.00)},
  MRNUMBER = {0098328 (20 \#4789)},
MRREVIEWER = {F. I. Mautner},
}

\bib{Meyer:Equivariant}{article}{
  author={Meyer, Ralf},
  title={Equivariant Kasparov theory and generalised homomorphisms},
  journal={\(K\)\nobreakdash-Theory},
  volume={21},
  date={2000},
  number={3},
  pages={201--228},
  issn={0920-3036},
  review={\MRref{1803228}{2001m:19013}},
  doi={10.1023/A:1026536332122},
}

\bib{Meyer:generalised_Fixed}{article}{
  author={Meyer, Ralf},
  title={Generalized fixed point algebras and square-integrable group actions},
  journal={J. Funct. Anal.},
  volume={186},
  date={2001},
  number={1},
  pages={167--195},
  issn={0022-1236},
  review={\MRref{1863296}{2002j:46086}},
  doi={10.1006/jfan.2001.3795},
}


\bib{Packer-Raeburn:Stabilisation}{article}{
  author={Packer, Judith A.},
  author={Raeburn, Iain},
  title={Twisted crossed products of $C^*$\nobreakdash-algebras},
  journal={Math. Proc. Cambridge Philos. Soc.},
  volume={106},
  date={1989},
  pages={293--311},
  review={\MRref{1002543}{90g:46097}},
  doi={10.1017/S0305004100078129},
}


\bib{Quigg:Duality_twisted}{article}{
    AUTHOR = {Quigg, John C.},
     TITLE = {Duality for reduced twisted crossed products of $C^*$\nobreakdash-algebras},
   JOURNAL = {Indiana Univ. Math. J.},
    VOLUME = {35},
      YEAR = {1986},
    NUMBER = {3},
     PAGES = {549--571},
      ISSN = {0022-2518},
  review={\MRref{855174}{88d:46122}},
   doi={10.1512/iumj.1986.35.35029},
}

\bib{Quigg}{article}{
    AUTHOR = {Quigg, John C.},
     TITLE = {Landstad duality for {$C^*$}-coactions},
   JOURNAL = {Math. Scand.},
  FJOURNAL = {Mathematica Scandinavica},
    VOLUME = {71},
      YEAR = {1992},
    NUMBER = {2},
     PAGES = {277--294},
      ISSN = {0025-5521},
     CODEN = {MTSCAN},
   MRCLASS = {46L55 (22D25)},
  MRNUMBER = {1212711 (94e:46119)},
MRREVIEWER = {Robert J. Archbold},
}

\bib{Raeburn-Williams:Morita_equivalence}{book}{
  author={Raeburn, Iain},
  author={Williams, Dana P.},
  title={Morita equivalence and continuous-trace $C^*$\nobreakdash-algebras},
  series={Mathematical Surveys and Monographs},
  volume={60},
  publisher={American Mathematical Society},
  place={Providence, RI},
  date={1998},
  pages={xiv+327},
  isbn={0-8218-0860-5},
  review={\MRref{1634408}{2000c:46108}},
}

\bib{Rieffel:Proper}{article}{
  author={Rieffel, Marc A.},
  title={Proper actions of groups on $C^*$\nobreakdash -algebras},
  conference={ title={Mappings of operator algebras}, address={Philadelphia, PA}, date={1988}, },
  book={ series={Progr. Math.}, volume={84}, publisher={Birkh\"auser Boston}, place={Boston, MA}, },
  date={1990},
  pages={141--182},
  review={\MRref {1103376}{92i:46079}},
}

\bib{Rieffel:Integrable_proper}{article}{
  author={Rieffel, Marc A.},
  title={Integrable and proper actions on $C^*$\nobreakdash -algebras, and square-integrable representations of groups},
  journal={Expo. Math.},
  volume={22},
  date={2004},
  number={1},
  pages={1--53},
  issn={0723-0869},
  review={\MRref {2166968}{2006g:46108}},
}

\bib{Rieffel:Deformation}{article}{
    AUTHOR = {Rieffel, Marc A.},
     TITLE = {Deformation quantization for actions of {${\bf R}^d$}},
   JOURNAL = {Mem. Amer. Math. Soc.},
  FJOURNAL = {Memoirs of the American Mathematical Society},
    VOLUME = {106},
      YEAR = {1993},
    NUMBER = {506},
     PAGES = {x+93},
      ISSN = {0065-9266},
     CODEN = {MAMCAU},
   MRCLASS = {46L87 (58B30 81R50 81S10)},
  MRNUMBER = {1184061 (94d:46072)},
MRREVIEWER = {Arne Schirrmacher},
       DOI = {10.1090/memo/0506},
       URL = {http://dx.doi.org/10.1090/memo/0506},
}
\bib{Rief-K}{article}{
    AUTHOR = {Rieffel, Marc A.},
     TITLE = {{$K$}-groups of {$C^*$}-algebras deformed by actions of
              {${\bf R}^d$}},
   JOURNAL = {J. Funct. Anal.},
  FJOURNAL = {Journal of Functional Analysis},
    VOLUME = {116},
      YEAR = {1993},
    NUMBER = {1},
     PAGES = {199--214},
      ISSN = {0022-1236},
     CODEN = {JFUAAW},
   MRCLASS = {46L55 (19K99 46L80 46L87 46L89)},
  MRNUMBER = {1237992 (94i:46088)},
MRREVIEWER = {Judith A. Packer},
       DOI = {10.1006/jfan.1993.1110},
       URL = {http://dx.doi.org/10.1006/jfan.1993.1110},
}

  \bib{Rief-Q}{incollection}{
    AUTHOR = {Rieffel, Marc A.},
     TITLE = {Quantization and {$C^\ast$}-algebras},
 BOOKTITLE = {{$C^\ast$}-algebras: 1943--1993 ({S}an {A}ntonio, {TX},
              1993)},
    SERIES = {Contemp. Math.},
    VOLUME = {167},
     PAGES = {66--97},
 PUBLISHER = {Amer. Math. Soc.},
   ADDRESS = {Providence, RI},
      YEAR = {1994},
   MRCLASS = {46L60 (81S10)},
  MRNUMBER = {1292010 (95h:46108)},
MRREVIEWER = {Aldo J. Lazar},
       DOI = {10.1090/conm/167/1292010},
       URL = {http://dx.doi.org/10.1090/conm/167/1292010},
}

\bib{Vaes-Vainerman:Extension_of_lcqg}{article}{
  author={Vaes, Stefaan},
  author={Vainerman, Leonid},
  title={Extensions of locally compact quantum groups and the bicrossed product construction},
  journal={Adv. Math.},
  volume={175},
  date={2003},
  number={1},
  pages={1--101},
  issn={0001-8708},
  review={\MRref{1970242}{2004i:46103}},
  doi={10.1016/S0001-8708(02)00040-3},
}

  \end{biblist}
\end{bibdiv}
\end{document}